\documentclass[11pt]{amsart}
\usepackage[english]{babel}
\usepackage{graphics}
\usepackage{amsfonts}
\usepackage{amsmath}
\usepackage{amssymb}
\usepackage{mathrsfs}
\usepackage{enumerate}
\usepackage{cite}
\usepackage[colorlinks=true,urlcolor=blue,
citecolor=red,linkcolor=blue,linktocpage,pdfpagelabels,
bookmarksnumbered,bookmarksopen]{hyperref}

\usepackage{esint}
\usepackage{epsfig}
\usepackage{amscd}
\usepackage{amsmath}
\usepackage{graphicx}
\usepackage{newlfont}
\usepackage{latexsym}
\usepackage{amsthm}

\usepackage{color}

\newtheorem{theorem}{Theorem}[section]
\newtheorem{lemma}[theorem]{Lemma}

\newtheorem{proposition}[theorem]{Proposition}

\theoremstyle{remark}
\newtheorem{remark}[theorem]{Remark}

\numberwithin{equation}{section}

\theoremstyle{definition}
\newtheorem{definition}[theorem]{Definition}

\newcommand{\R}{\mathbb{R}}

\newcommand{\Ci}{\mathcal{C}}

\newcommand{\lapl}{\Delta}

\makeindex
\begin{document}

\title{Symmetry in the composite plate problem}

\author[F.\ Colasuonno]{Francesca Colasuonno}
\author[E.\ Vecchi]{Eugenio Vecchi}
\address[F.\ Colasuonno]{Dipartimento di Matematica \newline\indent
	Universit\`a di Bologna \newline\indent
	Piazza di Porta S. Donato 5, 40126, Bologna, Italy}
\email{francesca.colasuonno@unibo.it}

\address[E.\ Vecchi]{Dipartimento di Matematica \newline\indent
	Sapienza Universit\`a di Roma,
	P.le Aldo Moro 5, 00185, Roma, Italy}
\email{vecchi@mat.uniroma1.it}
\thanks{{\bf Acknowledgments.} The authors are indebted to Prof. Sagun Chanillo for having suggested the problem
and for his valuable advice. The authors thank also Prof. Bruno Franchi for many
fruitful discussions and his support.\\
F.C. and E.V. are supported by {\em Gruppo Nazionale per l'Analisi Ma\-te\-ma\-ti\-ca, la Probabilit\`a e le loro Applicazioni} (GNAMPA) of the {\em Istituto Nazionale di Alta Matematica} (INdAM) and by University of Bologna, funds for selected research topics. F.C. is partially supported by the INdAM-GNAMPA Project 2017 ``Regolarit\`a delle soluzioni viscose per equazioni a derivate parziali non lineari degeneri''.
E.V. received funding from the People Programme (Marie Curie Actions) of the European Union's Seventh
	Framework Programme FP7/2007-2013/ under REA grant agreement No.\ 607643 (Grant MaNET `Metric Analysis for Emergent Technologies'), and was partially supported by the INdAM-GNAMPA Project 2017 ``Problemi nonlocali e degeneri nello spazio Euclideo''.}



\date{\today}

\begin{abstract}
In this paper we deal with the {\it composite plate problem}, namely the following optimization eigenvalue problem 
$$
\inf_{\rho \in \mathrm{P}} \inf_{u \in \mathcal{W}\setminus\{0\}} \frac{\int_{\Omega}(\Delta u)^2}{\int_{\Omega} \rho u^2},
$$
where $\mathrm{P}$ is a class of admissible densities, $\mathcal{W}= H^{2}_{0}(\Omega)$ for Dirichlet boundary conditions and
$\mathcal W= H^2(\Omega) \cap H^1_{0}(\Omega)$ for Navier boundary conditions. 
The associated Euler-Lagrange equation is a fourth-order elliptic PDE governed by the biharmonic operator $\Delta^2$.
In the spirit of \cite{CGIKO00}, we
study qualitative properties of the optimal pairs $(u,\rho)$. In particular,
we prove existence and regularity and we find the explicit expression of $\rho$.
When $\Omega$ is a ball, we can also prove uniqueness of the optimal pair,
as well as positivity of $u$ and radial symmetry of both $u$ and $\rho$.
\end{abstract}
\maketitle
\tableofcontents

\section{Introduction}

In a series of papers during the 2000's, many mathematicians 
(see e.g. \cite{CGIKO00,CGK,Sha,CK08,CKT,Chanillo13}) 
studied an eigenvalue optimization problem
that arises in Continuum Mechanics, usually referred to as {\it composite membrane
problem}. In physical terms, quoting \cite{CGIKO00}, it can be stated as follows:\smallskip

\noindent {\it Build a body of prescribed shape out of given materials (of varying densities)
in such a way that the body has a prescribed mass and so that the basic frequency
of the resulting membrane (with fixed boundary) is as small as possible.}\smallskip

This problem has a long history, without aiming at completeness, we just mention here the existence result proved
in \cite{Fried} and the qualitative results proved in \cite{CoxMc}.
We refer the interested reader to the monograph \cite{Henrot} and the
references therein for more results concerning this and related problems.

In mathematical terms, the composite membrane problem can be described in a variational way. Throughout the paper, for any measurable set $S \subset \Omega$, we denote by
$\chi_{S}$ its characteristic function and by $|S|$ its $n$-dimensional Lebesgue measure.
Let $\Omega \subset \mathbb{R}^{n}$ be a bounded
domain with Lipschitz boundary $\partial \Omega$,
$0 \leq h < H$ be two fixed constants, and $M \in [h \, |\Omega|, H \, |\Omega|]$.
Define the class of {\it admissible densities} as
\begin{equation}\label{Rho}
\mathrm{P} := \left\{ \rho : \Omega \to \mathbb{R}: \int_{\Omega}\rho(x) \, dx = M, \, h \leq \rho \leq H \, \mbox{ in }\Omega,\,\textrm{and}\,\rho\neq 0\mbox{ a.e. in }\Omega \right\}.
\end{equation}
The composite membrane problem is given by
\begin{equation*}
\Theta(h,H,M):= \inf_{\rho \in \mathrm{P}} \inf_{u \in H^{1}_{0}(\Omega)\setminus\{0\}} \dfrac{\int_{\Omega}|\nabla u|^2}{\int_{\Omega}\rho \, u^2},
\end{equation*}
\noindent and a couple $(u,\rho)$ which realizes the double infimum is called a {\it optimal pair}.
The first results proved in \cite{CGIKO00} and \cite{CGK} were however
obtained for a slightly more general eigenvalue optimization problem, which we
briefly describe: let $A\in[0,|\Omega|]$ and $\alpha >0$ be real 
numbers, and let 
$$\mathcal{S} := \left\{ S \subset \Omega : |S|=A \right\}$$
\noindent be the class of {\it admissible sets}. The minimization problem is
\begin{equation*}
\Lambda(\alpha,A):= \inf_{S \in \mathcal{S}} \inf_{u \in H^{1}_{0}(\Omega)\setminus\{0\}} \dfrac{\int_{\Omega}|\nabla u|^2 + \alpha \,\int_{\Omega}\chi_{S}u^2}{\int_{\Omega}u^2}.
\end{equation*}
\noindent In this case, we call {\it optimal pair} 
any couple $(u,S)$ which realizes the infimum. 
Let us spend a few words concerning the results proved in \cite{CGIKO00} for the last
problem. First of all, one is interested in proving existence of optimal pairs, 
and it can be done relying on a 
sort of {\it bathtub principle}, \cite{LL}. 
It is not possible however to expect uniqueness
of such solutions, unless assuming some kind of symmetry on the domain $\Omega$. 
We will come back to this aspect later on, because symmetry properties
will be at the core of our investigation along this paper. 
The second aspect concerns the regularity of the minimizers $u$
and the description of the optimal set $S$, which can be considered as
a free boundary. Concerning the regularity of the function $u$, one can
rely on classical elliptic regularity theory \cite{GT} and get the sharpest regularity.
A much more delicate issue is the study of the free boundary. 

More recently, in \cite{Chanillo13},
the author pointed out a close relation between the composite membrane problem
and a problem in conformal geometry (see Section \ref{confgeo} for more details) 
while an extension of the composite membrane problem to the case 
governed by the $p$-Laplacian operator can be found in \cite{Piel,CEP09,AC16}.

The aim of this paper is to study a fourth-order
analogue of the composite membrane problem, that can be called
{\it composite plate problem}. Similar problems have been recently investigated for instance in \cite{CEP06,ACP,Chen}, see also \cite{ColPr} for an analogous problem involving the polyharmonic operator.  
We now introduce our problem. Let $\Omega \subset \mathbb{R}^{n}$ be a bounded
domain with $C^4$-boundary $\partial \Omega$,
$0 \leq h < H$ be two fixed constants, and $M \in [h \, |\Omega|, H \, |\Omega|]$. Here we consider the dimensions $n\ge 2$, we refer to \cite{Banks, Chen} for the unidimensional case.
Define the class $\mathrm{P}$ of {\it admissible densities} $\rho$ as in \eqref{Rho} and let the functional space $\mathcal W$ be 
$$\mbox{either}\quad\mathcal{W}:= H^{2}_{0}(\Omega) \quad \textrm{or} \quad \mathcal{W}:= H^{2}(\Omega) \cap H^{1}_{0}(\Omega),$$
\noindent depending on the boundary conditions one wants to consider. 
The composite plate problem is given by
\begin{equation}\label{CP}
\tag{CP}
\Theta(h,H,M):= \inf_{\rho \in \mathrm{P}} \inf_{u \in \mathcal{W}\setminus\{0\}} \dfrac{\int_{\Omega}|\Delta u|^2}{\int_{\Omega}\rho \, u^2},
\end{equation}
\noindent and the associated Euler-Lagrange equation is the fourth-order problem
\begin{equation}\label{4EL}
\left\{ \begin{array}{rl}
           \Delta^2 u = \Theta \, \rho \, u, & \quad \textrm{in $\Omega$},\\
					          u =\Delta u = 0, & \quad \textrm{on $\partial \Omega$},
				\end{array}\right.						
\end{equation}
\noindent when $\mathcal{W}= H^{2}(\Omega) \cap H^{1}_{0}(\Omega)$,
and 
\begin{equation}\label{4ELd}
\left\{ \begin{array}{rl}
           \Delta^2 u = \Theta \, \rho \, u, & \quad \textrm{in $\Omega$},\\
					          u =\tfrac{\partial u}{\partial \nu} = 0, & \quad \textrm{on $\partial \Omega$},
				\end{array}\right.						
\end{equation}
\noindent when $\mathcal{W}= H^{2}_{0}(\Omega)$.
\begin{definition}
A couple $(u,\rho)\in\mathcal W\times\mathrm{P}$ which realizes the double infimum in \eqref{CP} is called  {\it CP-optimal pair}.
\end{definition}
\noindent As for its second-order analogue, this problem has a physical interpretation in 
Continuum Mechanics for inhomogeneous linear elastic plates (cf. Section \ref{Sec2}) and is related to the following more general variational problem. 
Let $\Omega \subset \R^n$ be as in \eqref{CP}, $\alpha >0$
and $A \in [0, |\Omega|]$ be real numbers. 
Let $\lambda_N=\lambda_{N}(\alpha,S)$
be the lowest eigenvalue of the following boundary value problem with 
Navier boundary conditions:
\begin{equation}\label{Nav}
\tag{$P_N$} 
\left\{ \begin{array}{rl}
\lapl^{2}u + \alpha \chi_{S} u = \lambda \, u, & \textrm{in } \Omega,\\
         u = \lapl u = 0, & \textrm{on } \partial \Omega,
\end{array}\right.\qquad \lambda\in\mathbb R,
\end{equation}
whose variational characterization is given by
$$
\lambda_{N}(\alpha,S) = \inf \left\{ R(u,\alpha,S) : u \in H^{2}(\Omega) \cap H^{1}_{0}(\Omega), u\not\equiv 0\right\},
$$
where 
$$R(u,\alpha,S):= \dfrac{\int_{\Omega}(\lapl u)^{2}dx + \alpha \, \int_{\Omega}\chi_{S}u^2 dx}{\int_{\Omega}u^2 dx}$$
denotes the Rayleigh quotient. \\
An analogous problem appears when considering Dirichlet boundary conditions.
Let $\lambda_D=\lambda_{D}(\alpha,S)$
be the lowest eigenvalue of the following Dirichlet boundary value problem:
\begin{equation}\label{Dir}
\tag{$P_D$} 
\left\{ \begin{array}{rl}
\lapl^{2}u + \alpha \chi_{S} u = \lambda \, u, & \textrm{in } \Omega,\\
         u = \tfrac{\partial u}{\partial \nu} = 0, & \textrm{on } \partial \Omega,
\end{array}\right.\qquad \lambda\in\mathbb R.
\end{equation}
The variational characterization of $\lambda_D$ is now given by
$$
\lambda_{D}(\alpha,S) = \inf \left\{ R(u,\alpha,S) : u \in H^{2}_{0}(\Omega), u\not\equiv 0\right\},
$$
\noindent with $R(u,\alpha,S)$ defined as above.
In both cases, we consider the following generalized problem 
\begin{equation}\label{Lambda}
\tag{G}
\Lambda_{j}(\alpha,A):= \inf_{S\in\mathcal S} \lambda_{j}(\Omega,\alpha,S)  \quad \textrm{for } j=N,D,
\end{equation}
where $\mathcal S=\{S \subset \Omega\,:\, |S|=A\}$ as above.

For notational ease, hereafter we will drop all subscripts $j$, $D$, $N$ of the eigenvalues.  

\begin{definition} A couple $(u,S)\in\mathcal W\times\mathcal S$ which realizes the double infimum in \eqref{Lambda} is called  {\it G-optimal pair}.
\end{definition}
\noindent We observe that the set $S$ is defined up to zero-measure sets.

Our first result for \eqref{Lambda} reads as follows.
\begin{theorem}\label{Theo1}
Let $\Omega \subset \R^n$ be a bounded domain with
$C^{4}$-boundary $\partial \Omega$.
For any positive $\alpha >0$ and every $A \in [0, |\Omega|]$,
there exists a G-optimal pair $(u,S)$. 
Furthermore, every G-optimal pair $(u,S)$ satisfies
\begin{itemize}
 \item[(a)] $u \in C^{3,\gamma}(\overline{\Omega}) \cap W^{4,q}(\Omega)$, for every $\gamma \in (0,1)$ and $q \geq 1$;
 \item[(b)] there exists a non-negative number $t\geq 0$ such that $S = \{ u^2 \leq t\}$.
\end{itemize}
\end{theorem}
We stress that, due to the presence of a characteristic function in the equation in \eqref{Lambda}, $C^{3,\gamma}(\overline{\Omega})$ is the the sharpest regularity we can obtain for $u$, see Remark~\ref{sharp-reg}. \smallskip

We say that problem \eqref{Lambda} is a generalization of \eqref{CP} because there exists a positive number $\bar{\alpha}(A)$
such that the two problems are in one-to-one correspondence for
every $\alpha \in (0, \bar{\alpha}(A)]$. 
 The explicit form of the optimal set $S$ for \eqref{Lambda} allows in turn to give a complete description of the optimal density $\rho$ of \eqref{CP}, as stated in the following theorem. 
\begin{theorem}\label{Struct} Under the structural assumptions on \eqref{Lambda} and \eqref{CP}, the following properties hold.
\begin{itemize}
  \item[(a)] Let $(u,\rho)$ be a CP-optimal pair, then $\rho$ has the following form:
$$
\rho= h \, \chi_{S} + H \, \chi_{S^{c}},
$$
\noindent for a set of the form $S = \{u^2 \leq t\}$.
	\item[(b)] The pair $(u,\rho)$ is a CP-optimal pair with
				parameters $(h,H,M)$ if and only if $(u,S)$ is a G-optimal pair 
				for \eqref{Lambda} with parameters $(\alpha,A)$ given by
				\begin{equation}\label{Alpha}
				\alpha = (H-h) \Theta,
				\end{equation}
				\begin{equation}\label{A}
				A= \tfrac{H |\Omega|-M}{H-h}.
				\end{equation}
Moreover, the two minimal eigenvalues are related by
\begin{equation}\label{Lambda_Theta}
\Lambda = H \Theta.
\end{equation}
	\item[(c)] When $h$ and $H$ vary in their ranges, 
	the corresponding $\alpha$ takes value in $(0,\bar\alpha(A)]$ 
	if $A< |\Omega|$, and in $(0,\infty)$ if $A=|\Omega|$. 
	In the first case, the value $\bar\alpha(A)$ occurs when $h=0$. 
\end{itemize}
\end{theorem}

The physical interpretation of Theorem \ref{Struct} is that the plate can be made only out of two materials, whose densities are given by the constants $h$ and $H$. Moreover, the denser material is farther from the boundary $\partial\Omega$. 
We mentioned that there are two main issues in the composite membrane
problem, namely {\it symmetry and symmetry breaking phenomena} 
and {\it regularity of the free boundary} of the generalized problem.
The same lines of investigation arise naturally in our context.
As a first step, we will study positivity and symmetry
properties of optimal pairs when $\Omega$ is a ball $B$.  
The assumption $\Omega=B$ could apparently be very restrictive, especially when
compared with the results available for the composite membrane problem. The main reason
behind this request can be roughly explained as follows. 
Symmetry properties of solutions of second-order elliptic equations 
can be proved by means of the {\it moving plane method} introduced
by Serrin in \cite{Serrin}, as a refinement of the {\it reflection principle} of Aleksandrov \cite{Aleksandrov1,Aleksandrov2}.
One of the main ingredients of this technique is the maximum principle.
The situation changes completely when dealing with fourth-order elliptic
equations. For example, symmetry and monotonicity results of Gidas-Ni-Nirenberg-type \cite{GNN} for semilinear biharmonic problems cannot hold even in the ball if the nonlinearity does not have the right sign, cf. \cite{Swe,BGW}. 
Moreover, there is a striking difference between Dirichlet and Navier boundary conditions. Indeed, for Navier 
it is possible to reduce the fourth-order equation to a second-order elliptic
system, where one recovers the main properties holding in the scalar case,
we refer to \cite{GGS} and references therein for a comprehensive survey of existing results on the topic.  In particular, the first eigenfunction of $\Delta^2$ with Navier boundary conditions is not sign-changing, while the same result does not hold in general domains under Dirichlet boundary conditions, cf. \cite{GrSw}. A second difficulty arises due to the fact that higher-order Sobolev spaces $W^{2,p}(\Omega)$ are not invariant under symmetric rearrangements, i.e., $u\in W^{2,p}(\Omega)$ does not imply that its symmetric rearrangement $u^*$ belongs to $W^{2,p}(\Omega)$, see \cite{Cianchi,GGS}.  
Nevertheless, there are instances where it is possible to bypass the structural problems appearing in the fourth-order context, e.g. \cite{Troy81,FGW}.
\smallskip

Let us now briefly describe our specific case.\\ 
For Navier boundary conditions, it is possible to rewrite
\eqref{4EL} as the second-order elliptic system
\begin{equation}\label{NSys}
\left\{ \begin{array}{rl}
          -\Delta u = v, & \quad \textrm{in $B$},\\
					-\Delta v = \Theta \, \rho \, u, & \quad \textrm{in $B$},\\
					u = v = 0, & \quad  \textrm{on $\partial B$}.
				\end{array}\right.					
\end{equation} 
Symmetry results for second-order elliptic systems on balls are available in the literature,
starting from the results by Troy \cite{Troy81} where the author considers $C^2$-solutions of the following system of PDE's 
\begin{equation}\label{Troy}
\left\{ \begin{array}{rl}
         - \Delta u_i = f_{i}(u_1, \ldots, u_n), & \quad \textrm{in $B$},\\
				  u_i >0, & \quad \textrm{in $B$},\\
				  u_i = 0,  & \quad \textrm{on $\partial B$},					
				\end{array}\right.	\quad i = 1, \ldots, m.
\end{equation} 
The nonlinearities $f_i$ are supposed to be of class $C^1$ and non-decreasing as functions of $u_j$ for every $j \neq i$. 
It is clear that once you fix an optimal configuration, \eqref{NSys}
becomes a cooperative elliptic system which
presents a non-autonomous right hand side $g(x,u)=\rho(x)u$ with no a priori
symmetry assumptions on the first entry.
Hence, it does not
satisfy the same assumptions of \eqref{Troy}, due to
the expression of $\rho$ which is the sum of two characteristic functions, and
therefore not smooth enough to allow the existence of
classical solutions. In particular, this implies that we 
deal with weak solutions in the appropriate Sobolev space. 
Despite these differences, the very specific structure of \eqref{NSys}, combined with the special form of the optimal $\rho$, allows to
adapt the proof of Troy even in our case, yielding the symmetry of the
weak solutions of \eqref{NSys}.\\
\indent The situation is in general more complicated when dealing with
Dirichlet boundary conditions, since much less symmetry results are available in the literature. 
Nevertheless, when $\Omega = B$, Ferrero, Gazzola and Weth in \cite{FGW} 
prove the radial symmetry for minimizers of subcritical Sobolev
inequalities, by means of {\it polarization}, cf. Section~\ref{Sec5}. This technique was introduced by Brock and Solynin in \cite{BrockSolynin} to avoid rearrangements methods. Indeed, since higher-order Sobolev spaces
are not closed under symmetrization, in those spaces it is not possible to have estimates of the form
$$\|\Delta u^{\ast}\|_{L^2} \leq \|\Delta u\|_{L^2},$$
useful in proving that the infimum of the Rayleigh ratio is achieved at a radial symmetric function. 
Again, the method used in \cite{FGW} exploits the continuity of the nonlinearity there involved, while in our case we cannot rely on such regularity. Here the fact that $\Theta \rho u$ can be regarded as a non-decreasing function of $u$ plays a crucial role in adapting the polarization technique in \cite{FGW} to get the desired symmetry.
\smallskip 

Throughout the paper we write {\it increasing} and {\it decreasing} meaning the {\it strict} monotonicity property.  

We state below our main result when $\Omega$ is a ball $B$. Without loss of generality we take $B:=\{x\in\mathbb R^n\,:\, |x|<1\}$. 
\begin{theorem}\label{Theo-ball}
Let $\Omega=B$, then there exists a unique CP-optimal pair $(u,\rho)$. 
Furthermore, $u$ is positive, radial, and radially decreasing. The set $S$ for which $\rho=h\chi_S + H\chi_{S^c}$ is the unique shell region $\{x\,:\, r(A)<|x|<1\}$ of measure 
$A$ (i.e., $r(A)>0$ is the unique positive constant for which  $| \{x\,:\, r(A)<|x|<1\}|=A$).  
\end{theorem}

\indent We point out that if $(u,\rho)$ is a CP-optimal pair then, for every $\mu \in \mathbb{R}\setminus \{0\}$, 
$(\mu u, \rho)$ is a CP-optimal pair as well, cf. Remark \ref{case_h=0}. This means that uniqueness of CP-optimal pairs in Theorem \ref{Theo-ball}
has to be intended up to a multiplicative constant in $u$. \\
\indent A few comments on the existing literature are now in order. Theorem \ref{Theo-ball} is morally stated in \cite[Remark 11.4.2]{Henrot} as a direct consequence of the technique introduced by P\'olya-Szeg\"o in \cite[Section F.5]{PolSz} for the biharmonic Faber-Krahn problem. Nevertheless, it came out that this technique is not suitable for higher-order problems
\footnote{Quoting \cite[p. 72]{GGS}: {\it  (...) $u^*$ may not be twice weakly differentiable even if $u$ is very smooth. In their monograph, P\'olya-Szeg\"o \cite[Section F.5]{PolSz} claim that they can extend the Faber-Krahn result to the Dirichlet biharmonic operator among
domains having a first eigenfunction of fixed sign. Not only this assumption does
not cover all domains (...) but also their argument is not correct.
They deal with the Laplacian of a symmetrised smooth function and implicitly claim
that it belongs to $L^2$, which is false in general. (...)
This shows that standard symmetrisation methods are not available for higher order
problems.}}. Furthermore, a symmetry result for a problem somehow related to ours is stated in \cite{ACP}.
\smallskip

The paper is organized as follows. In Section \ref{Sec2} we describe the physical interpretation of the problem and we recall some known results that are useful in the rest of the paper. In Section \ref{Sec3} we prove Theorem \ref{Theo1} and study the dependence of $\Lambda$ on the parameters $\alpha$ and $A$. In Section \ref{Relation}, we show the relation between the two problems \eqref{Lambda} and \eqref{CP} proving Theorem \ref{Struct}, while in Section~\ref{Sec5} we prove Theorem \ref{Theo-ball}. Finally, in Section \ref{confgeo}, we present an application to a problem in conformal geometry.


\section{Preliminaries and known results}\label{Sec2}
We start this section with a detailed physical interpretation of problems \eqref{4EL} and \eqref{4ELd}. As already mentioned in the introduction, when $n=2$ these problems are related to Continuum Mechanics for inhomogeneous linear elastic plates. Plates are plane structural elements with a 
small thickness compared to the planar dimensions.
For a transversely loaded plate without axial deformations, 
the governing equation is given by the Germain-Lagrange equation
$$\Delta^2 u(x)=\frac{q(x)}{D},$$
where $u(x)$ is the transverse displacement of the plate at $x$, 
$q$ is the imposed stress, which is supposed to be 
a distributed external load that is normal to the mid-surface, 
and $D$ is the flexural rigidity, supposed to be constant. The constant $D$ depends 
on the material of the plate and its geometry as follows
$$D=\frac{Eh^3}{12(1-\nu)},$$
where $E$ is the Young modulus, $h$ the thickness of the plate, 
and $\nu$ is the Poisson coefficient. In particular, 
the units of $D$ are $[D]=\textrm{N}\cdot \textrm{m}$. 
We can always write the stress $q$ as
$$q=\rho\cdot a,$$
where $\rho$ is the surface density and $a$ an acceleration. 
We suppose that the acceleration is proportional to the displacement 
$$q=\rho\cdot a=\beta\rho u,$$
with $[\beta]=\mathrm{s}^{-2}$. Therefore, if we include all the 
constants in $\Theta$ in the equation in \eqref{4EL} (or \eqref{4ELd}), 
we get $\Theta=\beta/D$ and its units are $[\Theta]=\textrm{kg}^{-1}\textrm{m}^{-2}$, the same as the ones of an eigenvalue of $\Delta^2$ (i.e., $\mathrm{m}^{-4}$), divided by a surface density. 
Finally, Dirichlet boundary conditions are meant to describe a clamped plate, while Navier boundary conditions, when $\Omega$ is a polygonal domain in $\mathbb R^2$, describe a hinged plate, cf. \cite{Sweers}. 
\medskip 

Throughout the paper, unless differently stated, 
$\Omega \subset \R^n$, with $n\ge2$, will denote a bounded domain (i.e., open and connected) 
with $C^4$-boundary $\partial \Omega$. This regularity assumption is needed to prove the regularity result in Theorem \ref{Theo1} up to the boundary, but it can be weakened if we look for less regular solutions, cf. Remark \ref{sharp-reg}. We assume throughout the paper that $0 \in \Omega$, 
this can be done without loss of generality, 
since the problems we are considering are invariant under translation.
Furthermore, with a slight abuse of notation, we denote  
\begin{equation*}
\left\{ f < t \right\} := \left\{ x \in \Omega: \, f(x) < t \right\},
\end{equation*}
\noindent and analogously for $\{f \leq t\}$.

\medskip

As already mentioned in the introduction, we work with two 
Sobolev spaces: we use $H^{2}_{0}(\Omega)$ for the
{\it clamped plate}, and $H^{2}(\Omega)\cap H^{1}_{0}(\Omega)$ for
the {\it hinged plate}. Some of our results will be proved
in the same way either for the hinged plate or the clamped one.
In these cases, to simplify the notation, 
we will denote both spaces by $\mathcal{W}$. In both cases, we consider the space equipped with the
following norm 
$$\|u\|_{\mathcal{W}}^2:= \int_{\Omega}(\Delta u)^2 \, dx, \quad u \in \mathcal{W}$$
which is equivalent to the standard Sobolev one. The proof of the equivalence in $H^2_0(\Omega)$ relies on the Poincar\'e and Calder\'on-Zygmund inequalities, see for instance \cite[Chapter 2.7]{GGS}, 
while in $H^2(\Omega)\cap H^1_0(\Omega)$ it is a consequence of the equivalence in $H^2_0(\Omega)$ and the continuous embedding $H^2_0(\Omega)\hookrightarrow H^2(\Omega)\cap H^1_0(\Omega)$.	  
We stress that $\mathcal{W}$ endowed with $\|\cdot\|_{\mathcal{W}}$ is
a Hilbert space. There is a huge literature dealing with best constants of
the critical embeddings of these spaces, e.g. \cite{VdV}. We refer
to the monograph \cite{GGS} for a comprehensive introduction to
the subject.

We recall here two classical embedding theorems that will be useful in what follows.
\begin{theorem}
\label{EmbLq}
Let $\Omega \subset \R^n$ be an open and bounded set with
Lipschitz boundary $\partial \Omega$. Let $1\leq p < +\infty$ and let 
$m \in \mathbb{N}^{+}$. Then, the following continuous embeddings hold 
\begin{equation*}
W^{m,p}(\Omega) \hookrightarrow L^{q}(\Omega)\quad\mbox{for any }q\in
\begin{cases}
[1, \frac{np}{n-mp}],&\quad\mbox{if }n>mp,\\
[1,\infty),&\quad\mbox{if }n\le mp.
\end{cases}
\end{equation*}
\end{theorem}
An improvement of Theorem \ref{EmbLq} when $n<mp$ is given by
the following 
\begin{theorem}[Theorem 2.6 of \cite{GGS}]\label{EmbCk}
Let $\Omega \subset \R^n$ be an open and bounded set with Lipschitz boundary
$\partial \Omega$. Let $1\leq p < +\infty$ and let 
$m \in \mathbb{N}^{+}$ and assume that there exists
$k \in \mathbb{N}$ such that $n < (m-k)p$. Then
$$W^{m,p}(\Omega) \hookrightarrow C^{k,\gamma}(\overline{\Omega}), 
\quad \textrm{for every } \gamma \in \big(0, m-k-\dfrac{n}{p}\big] \cap (0,1),$$
with compact embedding if $\gamma < m-k-\tfrac{n}{p}$.
\end{theorem}

The following maximum principle for a forth-order problem set in a ball will be useful in Section \ref{Sec5}.

\begin{lemma}[Lemma 1 of \cite{FGW}]\label{max_principle}
Let $\Omega=B:=\{x\in\mathbb R^n\,:\, |x|<1\}$ and $\mathcal C^+:=\{w\in \mathcal{W}\,:\,w\ge0\mbox{ a.e. in }B\}$. Assume that $u\in \mathcal{W}(B)$ is such that 
$$\int_B \Delta u\Delta v\ge0\quad\mbox{for every }v\in\mathcal C^+,$$
then $u\in\mathcal C^+$. Moreover, either $u\equiv 0$ or $u>0$ a.e. in $B$. 
\end{lemma}

\begin{remark}
We stated Lemma \ref{max_principle} just in the case of the ball $B$
but, for Navier boundary conditions, it is actually possible to consider more general
domains $\Omega$ with Lipschitz boundary, and the proof is 
precisely the same as in \cite[Lemma 1]{FGW}.
\end{remark}
\medskip

We introduce now some notation, definitions and preliminary results on the polarization of a function. This technique will be useful when dealing with the symmetry properties in the problem with Dirichlet boundary conditions.

\begin{definition}
Let $\mathcal H\subset\mathbb R^n$ be a half-space 
with boundary $\partial \mathcal H$, and for every $x\in\mathbb R^n$, 
let $\bar x$ denote the reflection of $x$ with respect to $\partial \mathcal H$. 
For every function $v:\mathbb R^n\to\mathbb R$, we define its 
{\it polarization relative to $\mathcal H$} as $v_{\mathcal H}:\mathbb R^n\to\mathbb R^n$ such that 
$$
v_{\mathcal H}(x):=\left\{\begin{array}{rl}
\max\{v(x),v(\bar x)\},&\quad \mbox{if }x\in\mathcal H,\\
\min\{v(x),v(\bar x)\},&\quad\mbox{if }x\in\mathbb R^n\setminus\mathcal H.
\end{array}\right.
$$
\end{definition}

It is straightforward to check that polarization preserves 
continuity and moreover, if $v$ is a compact supported continuous 
function, then also $v_\mathcal H\in C_c(\mathbb R^n)$. 
Furthermore, every polarization preserves the $L^p$-norms ($1\le p\le +\infty$) 
and the following pointwise identity holds
\begin{equation}\label{identity}
v(x)+v(\bar x)=v_{\mathcal H}(x)+v_{\mathcal H}(\bar x)\quad\mbox{for every }x\in\mathbb R^n.
\end{equation}

\begin{proposition}[Ex. 2.4 of \cite{Burchard}, \cite{LL}]\label{equimisurabilita}
Let $\mathcal H\subset \mathbb R^n$ be a 
half-space and $f\in L^1(\mathbb R^n)$ be non-negative, then 
$$
|\{x\,:\, f(x)>s\}| = |\{x\,:\, f_\mathcal H(x)>s\}| \quad \mbox{for every }s>0.
$$
\end{proposition}

We will use the following characterization for radial, radially non-increasing functions. 

\begin{lemma}[Lemma 6.3 of \cite{BrockSolynin}]\label{Lemma-characterization}
A function $v\in C_c(\mathbb R^n)$ is radial and radially 
non-increasing if and only if $v=v_{\mathcal H}$ for every 
 half-space $\mathcal H\subset\mathbb R^n$ such that 
$0\in\mathrm{int}(\mathcal H)$.
\end{lemma}

\begin{lemma}[Lemma 3 of \cite{FGW}]\label{Lemma-G}
Let $\mathcal{H}$ be a half-space such that 
$0\in\mathrm{int}(\mathcal H)$ and $G=G(x,y)$ the Green 
function of $\Delta^2$ in $B$ relative to Dirichlet boundary 
conditions. Then, for every $x,\,y\in\mathcal H$, $x\neq y$ 
the following inequalities hold
\begin{itemize}
\item[(i)] $G(x,y)\ge\max\{G(x,\bar y),G(\bar x,y)\}$;
\item[(ii)] $G(x,y)-G(\bar x,\bar y)\ge |G(x,\bar y)-G(\bar x,y)|$.
\end{itemize}
Moreover, if $x,\,y\in\mathrm{int}(B\cap\mathcal H)$, 
the inequalities in (i) and (ii) are strict.
\end{lemma}

\section{Proof of Theorem \ref{Theo1}}\label{Sec3}
The aim of this section is to prove Theorem \ref{Theo1}. 
Besides the regularity of the solutions of either \eqref{Nav} or \eqref{Dir},
Theorem \ref{Theo1} provides an explicit description of
the optimal set $S$ as a sub-level set of $u^2$. Once
established the connection between \eqref{Lambda} and \eqref{CP},
the knowledge of the optimal set $S$ will be crucial to
provide also a description of any optimal density $\rho$,
which in turn will play a crucial role in the study of
the symmetry properties of $u$.

Before proving regularity in our case, we recall the following result in a more general setting.  

\begin{proposition}[Theorem 2.20 of \cite{GGS}]\label{Reg}
Let $\Omega \subset \R^n$ be a bounded domain, with
$C^{4}$-smooth boundary $\partial\Omega$, and let $f \in L^p(\Omega)$ for $p\in (1,\infty)$.
Then 
\begin{equation}\label{general_eq}
\Delta^2 u=f\quad\mbox{in }\Omega,
\end{equation}
coupled either with Dirichlet or with Navier boundary conditions, admits a unique strong solution\footnote{For {\it strong solution} we mean a function $u$ that satisfies the equation \eqref{general_eq} almost everywhere. 
} in $W^{4,p}(\Omega)$.
\end{proposition}

Now, we are ready to prove Part $(a)$ of Theorem \ref{Theo1}.
\begin{proof}[$\bullet$ Proof of Part (a)]
Let us go back to the fourth-order PDE
$$\lapl^2 u + \alpha \chi_{S} u = \lambda u\quad \textrm{in } \Omega$$
and define the function 
\begin{equation}\label{f}
f(x) := \lambda u(x) - \alpha \chi_{S}(x) u(x).
\end{equation}
Since $u$ is a weak solution of either \eqref{Nav} or \eqref{Dir},
it holds that $u \in H^{2}(\Omega)$. This implies that
$f \in L^{2}(\Omega)$, and so $f \in L^{p}(\Omega)$ 
for every $p \in [1,2]$, being $\Omega$  bounded. 
Therefore, by
Proposition \ref{Reg}, we know that $u \in W^{4,p}(\Omega)$ for every $p \in [1,2]$.\\
Now, by Theorem \ref{EmbLq}
$$H^{4}(\Omega)\subset L^{q}(\Omega)\quad\mbox{for all }q\in\begin{cases}
[1,2^*],\quad&\mbox{if }n>8,\\
[1,2^*),&\mbox{if }n\le8,
\end{cases}
\quad
2^{\ast}:= \left\{ \begin{array}{rl} 
                        \tfrac{2n}{n-8}, & \textrm{if } n >8,\\
												+\infty, & \textrm{if } n \leq 8. 
                     \end{array}\right.
$$	
If $2^\ast=+\infty$, then $u$ and $f$ belong to $L^p(\Omega)$ 
for all $p\in[1,\infty)$. By Proposition \ref{Reg}, $u\in W^{4,p}(\Omega)$ 
for all $p\in[1,\infty)$. In particular, $u\in W^{4,p}(\Omega)$ 
for all $p>n$, hence $u\in C^{3,\gamma}(\overline \Omega)$ 
for all $\gamma\in(0,1)$, by Theorem \ref{EmbCk}.  
If $2^\ast<\infty$, we use a bootstrap argument.  			
For every $j\in \mathbb{N}$, we define
\begin{equation*}
2^{\ast}_{j}:= \left\{ \begin{array}{rl} 
                        \tfrac{2n}{n-8j}, & \textrm{if } n >8j,\\
												+\infty, & \textrm{if } n \leq 8j. 
                     \end{array}\right.
\end{equation*}										
It is straightforward to verify, by induction on $j\ge1$,
 that $2^\ast_{j+1}=(2^\ast_{j})^\ast$.
Since $u\in L^{2^\ast}(\Omega)$, also $f\in L^{2^\ast}(\Omega)$ and, 
again by Proposition \ref{Reg}, we have $u\in W^{4,2^\ast}(\Omega)$. 
Iterating the application of both Proposition \ref{Reg} and  
Theorem \ref{EmbLq} $j$-times, as long as $2^*_j<\infty$, 
we get that $u \in W^{4,2^{\ast}_{j-1}}(\Omega)$ and
$$W^{4,2^{\ast}_{j-1}}(\Omega) \subset L^{2^{\ast}_j}(\Omega).$$
Now, for every $n\in \mathbb{N}$, there exists $\bar j \in \mathbb{N}$ 
such that $n \le 8\bar j$ and so, $2^\ast_{\bar j}=+\infty$. 
After $\bar j$ iterations, we can conclude by using 
Theorem \ref{EmbCk}, as already done in the  case $2^\ast=\infty$.
\end{proof} 

\begin{remark}\label{sharp-reg}
The regularity of $u$ cannot be improved up
to $C^{4}(\Omega)$, at least in the more
relevant cases when $\emptyset\neq S \subsetneq \Omega$, due to the presence of the characteristic function.\\
We want also to stress another fact: 
from the modeling point of view, it is more reasonable 
to work with a Lipschitz boundary $\partial \Omega$.
In this case, we can obtain the same regularity
result of Theorem~\ref{Theo1}-$(a)$, but only in the {\it interior},
mainly due to the fact that the argument provided 
by \cite[Theorem~2.20]{GGS}
requires a smooth enough boundary. Therefore, if we restrict
our attention to interior regularity, we can use 
the same bootstrap argument presented in the proof
of Theorem~\ref{Theo1}-$(a)$ to prove that a weak solution $u$ of
\eqref{Nav} (or \eqref{Dir}) is such that
$$u \in W^{4,q}_{\textrm{loc}}(\Omega) \cap C^{3,\gamma}(\Omega)$$
for every $q \in [1, \infty)$ and for every $\gamma \in (0,1)$.
\end{remark}
\medskip


Let us prove now the existence of a G-optimal pair.
As for the regularity, the strategy of the proof of existence
is independent of the boundary
conditions imposed. Therefore, we will adopt
the compact notation $\mathcal{W}$ for the Sobolev
space over which we consider the infimum. \\
\indent We first prove an auxiliary result.
\begin{proposition}\label{min-pb}
Let $A \ge 0$ be a fixed non-negative constant,
$$\mathcal{A}:= \left\{ \eta:\Omega \to \mathbb{R} : 0\leq \eta \leq 1 \, \textrm{a.e. in } \Omega, \, \int_{\Omega}\eta =A \right\},$$
and $u \in \mathcal{W}$ such that $\|u\|_{L^{2}(\Omega)}=1$.
If we define the functional $I: \mathcal{A} \to \mathbb{R}$ as 
$$I(\eta) := \int_{\Omega}\eta(x) \, u^2(x) \, dx,$$
then the minimization problem 
\begin{equation}\label{infpb}
\inf_{\eta \in \mathcal{A}} I(\eta)
\end{equation}
admits a solution $\eta = \chi_{S}$,
with $S$ belonging to the following set
\begin{equation}\label{SetS}
\begin{aligned}&\mathcal S_t:=\Big\{S\subset\Omega\;:\;|S|=A,\,\{u^2 <t\} \subset S \subset \{u^2 \leq t\}\Big\},\\ &\mbox{where}\quad t:= \sup \{ s>0: |\{u^2 <s\}|<A\}.
\end{aligned}
\end{equation}
In particular, for every $\alpha>0$
\begin{equation}\label{Lambda=infeta}
\Lambda(\alpha,A)=\inf_{\eta \in \mathcal{A}} \inf_{u \in \mathcal{W}\setminus\{0\}} \dfrac{\int_{\Omega}(\lapl u)^2 + \alpha \int_{\Omega}\eta u^2}{\int_{\Omega}u^2}
\end{equation}
and the set $S$, that realizes $\Lambda$, belongs to $\mathcal S_t$.
\begin{proof}
We observe that for every set $S$ of measure $A$, 
its characteristic function $\chi_S$ belongs to $\mathcal A$. 
Hence, it is enough to prove that $I(\chi_{S}) \leq I(\eta)$ 
for every $S\subset\Omega$ satisfying \eqref{SetS} and 
for every $\eta \in \mathcal{A}$.
A simple splitting of the domain of integration yields
\begin{equation*}
\begin{aligned}
\int_{\Omega}u^{2} (\chi_{S}-\eta)dx &= \int_{\{u^2 <t\}}u^{2} (\chi_{S}-\eta)dx + \int_{\{u^2 >t\}}u^{2} (\chi_{S}-\eta)dx + \int_{\{u^2 =t\}}u^{2} (\chi_{S}-\eta) dx\\
&\leq t \, \int_{\{u^2 <t\}} (\chi_{S}-\eta)dx- t \int_{\{u^2 >t\}} \eta dx+t\int_{\{u^2 =t\}}(\chi_{S}-\eta) dx\\
&= t \left( \int_{\{u^2 <t\}} (\chi_{S}-\eta)dx+\int_{\{u^2 >t\}} (\chi_S-\eta) dx+\int_{\{u^2 =t\}}(\chi_{S}-\eta) dx\right)\\
&=t\int_\Omega(\chi_S-\eta)dx=0.
\end{aligned}
\end{equation*}
This closes the proof of the first part of the statement and easily gives
\begin{equation}\label{inf=inf}
\inf_{\eta \in \mathcal{A}} \inf_{u \in \mathcal{W}\setminus\{0\}} \dfrac{\int_{\Omega}(\lapl u)^2 + \alpha \int_{\Omega}\eta u^2}{\int_{\Omega}u^2}=\inf_{S\in\mathcal S_t} \inf_{u \in \mathcal{W}\setminus\{0\}} \dfrac{\int_{\Omega}(\lapl u)^2 + \alpha \int_{\Omega}\chi_S u^2}{\int_{\Omega}u^2}.
\end{equation}
Indeed, since 
$$\left\{ \chi_S\,:\, S\in\mathcal S_t \right\} \subseteq \mathcal{A},$$
we get  
$$\inf_{\eta \in \mathcal{A}} \inf_{u \in \mathcal{W}\setminus\{0\}} \dfrac{\int_{\Omega}(\lapl u)^2 + \alpha \int_{\Omega}\eta u^2}{\int_{\Omega}u^2} \leq \inf_{S\in\mathcal S_t} \inf_{u \in \mathcal{W}\setminus\{0\}} \dfrac{\int_{\Omega}(\lapl u)^2 + \alpha \int_{\Omega}\chi_S u^2}{\int_{\Omega}u^2}.$$
The opposite inequality follows directly from the previous computation.\\
\indent Now, on one hand, 
\begin{equation*}
\inf_{\eta \in \mathcal{A}} \inf_{u \in \mathcal{W}\setminus\{0\}} \dfrac{\int_{\Omega}(\lapl u)^2 + \alpha \int_{\Omega}\eta u^2}{\int_{\Omega}u^2} \le \Lambda(\alpha,A),
\end{equation*}
being $\chi_S\in\mathcal A$ for all $S$ of measure $A$.
\noindent On the other hand, by definition \eqref{Lambda} of $\Lambda$
\begin{equation*}
\Lambda(\alpha,A) \le \inf_{S\in\mathcal S_t} \inf_{u \in \mathcal{W}\setminus\{0\}} \dfrac{\int_{\Omega}(\lapl u)^2 + \alpha \int_{\Omega}\chi_S u^2}{\int_{\Omega}u^2}.
\end{equation*}
Together with \eqref{inf=inf}, this implies \eqref{Lambda=infeta} and concludes the proof.
\end{proof}
\end{proposition}

We are now ready to proceed with the proof of the
existence of a G-optimal pair.
In what follows we fix $\alpha> 0$ and $A\in[0,|\Omega|]$ and we simplify the notation
by writing
$$
\Lambda:= \Lambda(\alpha,A)\quad\mbox{and}\quad \lambda(S):= \lambda(\alpha,S)\textrm{ for every } S\subset \Omega.
$$

\begin{proof}[$\bullet$ Proof of existence] 
Let $(S_k)_{k}$ be a minimizing sequence, meaning $|S_k|=A$ for every $k \in \mathbb{N}$ and
$$\lambda(S_k) \longrightarrow \Lambda \quad \textrm{as } k \to \infty.$$
For every $k \in \mathbb{N}$, we consider 
a first eigenfunction 
$u_k \in \mathcal{W}$
of $\Delta^2 + \alpha \chi_{S_k}$. Without loss of generality, we can
assume that $\|u_k\|_{L^{2}(\Omega)}=1$ for every $k\in \mathbb{N}$.\\
Now, the sequences $(\chi_{S_k})_k \subset L^2(\Omega)$ 
and $(\lambda(S_k))_k$ are bounded. Keeping in mind that 
the norm used is $\|u\|_{\mathcal{W}}^{2}= \int_{\Omega}(\Delta u)^2$,
the previous considerations imply that
 $(u_k)_k$ is a bounded sequence in $\mathcal{W}$.
Since both the spaces $L^2(\Omega)$ and $\mathcal{W}$ are Hilbert spaces,
we can extract two sub-sequences, still denoted $(\chi_{S_k})_k$ and $(u_k)_k$, and we can find
two functions $\eta \in L^{2}(\Omega)$ and $u \in \mathcal{W}$, such that
$$
\begin{aligned}
&\chi_{S_k} \rightharpoonup  \eta \quad \textrm{in } L^{2}(\Omega),\quad \textrm{as } k\to \infty,\\
&u_k \rightharpoonup  u \quad\textrm{ in } \mathcal{W}, \quad\textrm{as } k\to \infty.
\end{aligned}
$$
Hence, up to a subsequence, we have the following:
\begin{itemize}
\item[(i)] $u_k \rightarrow u$ in $L^{2}(\Omega)$, as $k \to \infty$;\\
\item[(ii)] $\int_\Omega\chi_{S_k}u_k\psi dx \to\int_\Omega \eta u\psi dx$\, for every $\psi\in C^\infty_0(\Omega)$, as $k\to \infty$;\\
\item[(iii)]$\int_{\Omega}\eta(x)dx =A$. 
\end{itemize}
Indeed, (i) follows from the compact embedding $\mathcal{W} \hookrightarrow L^{2}(\Omega)$; 
(ii) follows from (i) and H\"older's inequality in the direct computation
$$
\begin{aligned}
\left|\int_\Omega\left(\chi_{S_k}u_k-\eta u\right)\psi dx\right|&\le 
\left|\int_\Omega \chi_{S_k}(u_k-u)\psi dx\right|+\left|\int_\Omega (\chi_{S_k}-\eta)u\psi dx\right|\\
&\le \|u_k- u\|_{L^2(\Omega)}\|\psi\|_{L^2(\Omega)}+\left|\int_\Omega (\chi_{S_k}-\eta)u\psi dx\right|\to 0
\end{aligned}
$$
for every $\psi\in C_0^\infty(\Omega)$.
To prove (iii) we argue as follows: since $\chi_{S_k} \rightharpoonup \eta$ in $L^{2}(\Omega)$
and $\Omega$ is bounded, we have in particular
$$
A=\int_{\Omega}\chi_{S_k}\cdot 1\, dx \rightarrow \int_{\Omega} \eta\cdot 1\, dx,
$$
which gives (iii) by uniqueness of the limit. \\
By definition, any pair $(u_k,S_k)$ satisfies 
\begin{equation}\label{k-eq}
\Delta^2 u_k + \alpha \chi_{S_k} u_k = \lambda_{S_k} u_k
\end{equation}
and so 
\begin{equation}\label{var-for}
\int_{\Omega}\Delta u_k \Delta \psi + \alpha \int_{\Omega} \chi_{S_k}u_k \psi = \lambda_{S_k}\int_{\Omega}u_k \psi \quad\mbox{for all }\psi \in C^{\infty}_{0}(\Omega).
\end{equation}
By previous remarks, we can pass to the limit in \eqref{var-for} as $k\to\infty$,
finding
$$\int_{\Omega}\Delta u \Delta \psi + \alpha \int_{\Omega} \eta u \psi = \Lambda\int_{\Omega}u \psi.$$
Integrating by parts, we recover the variational formulation
of the eigenvalue equation associated to $\Lambda$, which implies that $u\in \mathcal{W}$ solves the equation
\begin{equation}\label{weak-eq} 
\Delta^2 u + \alpha \eta u = \Lambda u, 
\end{equation}
in the weak sense. 
Now, the sets 
$$
\begin{aligned}
P_{a} &:= \left\{ w \in L^{2}(\Omega): w(x) \leq 1 \quad \textrm{for a.e. } x \in \Omega\right\}\\
P_{b} &:= \left\{ w \in L^{2}(\Omega): w(x) \geq 0 \quad \textrm{for a.e. } x \in \Omega\right\}
\end{aligned}
$$
are strongly closed in the $L^2$-topology and convex, then weakly closed. 
Since $(\chi_{S_k})_k\subset P_a\cap P_b$ and $\chi_{S_k}\rightharpoonup \eta$ in $L^2(\Omega)$,
$$0 \leq \eta(x) \leq 1 \textrm{ for a.e. } x \in \Omega.$$
Thus, $\eta\in\mathcal A$. 
In order to end the proof, we need
to show that we can replace the function $\eta$ with a characteristic
function of a suitable set $S \subset \R^n$ of measure $A$. To this aim,
let us multiply \eqref{weak-eq} by $u$ 
and let us integrate it over $\Omega$. Since by (i) we have
$\|u\|_{L^{2}(\Omega)}=1,$ it follows that
\begin{equation*}
\int_{\Omega}(\Delta u)^2 + \alpha \, \int_{\Omega} \eta u^2 = \Lambda.
\end{equation*}
By Proposition \ref{min-pb}, we have that there exists a set $S \subset \Omega$ satisfying \eqref{SetS} such that
$$\Lambda = \int_{\Omega}(\Delta u)^2 + \alpha \, \int_{\Omega} \eta u^2 \geq \int_{\Omega}(\Delta u)^2 
+ \alpha \, \int_{\Omega} \chi_{S} u^2.$$
Hence, from the definition of $\Lambda$ as an infimum, we have
that
$$\int_{\Omega}(\Delta u)^2 + \alpha \, \int_{\Omega} \chi_{S} u^2 = \Lambda,$$
and therefore the pair $(u,S)$ is a G-optimal pair.
\end{proof}\medskip

We can now give a precise description of the optimal set $S$ in terms of a sub-level of $u^2$.
 
\begin{proof}[$\bullet$ Proof of Part (b)]
Let $(u,S)$ be a G-optimal pair. 
By the proof of the existence result, we know that $S\in\mathcal S_t$, with $t$ defined as in \eqref{SetS}. Hence, it is enough to prove that $\mathcal{N}_t :=\{u^2=t\}\subset S$. 
Now, if $t>0$, $\mathcal{N}_t=\{u=\sqrt{t}\}\cup\{u=-\sqrt{t}\}$.
By \cite[Lemma~7.7]{GT}, we have that $\lapl^2 u = 0$ a.e. in
$\mathcal{N}_t$, being $u$ constant in both $\{u=\sqrt{t}\}$ and $\{u=-\sqrt{t}\}$. 
Therefore, the Euler-Lagrange equation associated
to \eqref{Lambda} reduces to
$$(\Lambda \cdot \mathrm{Id} - \alpha \chi_S)u = 0 \quad \textrm{a.e. in } \mathcal{N}_t.$$
Since $u \neq 0$ 
in $\mathcal{N}_t$,
this implies that $\Lambda \cdot \mathrm{Id} = \alpha \chi_S$ a.e. in $\mathcal{N}_t$, which yields
in turn $\mathcal{N}_t \subset S$, being $\Lambda, \alpha >0$ . This concludes the proof in the case $t>0$. 

If $t=0$, we have to prove that $\mathcal N_0=\{u=0\}\subset S$. 
By \eqref{SetS}, we know that $S\subset \{u=0\}$, thus $\chi_S u=0$ in $\Omega$ 
and the equation reduces to $\Delta^2 u=\Lambda u$ in $\Omega$. Thus, $\Lambda=\mu(\Omega)$, 
where $\mu(\Omega)$ is the first eigenvalue of $\Delta^2$ in $\Omega$ with either Navier or Dirichlet boundary conditions, and $u$ is the corresponding first eigenfunction. 
Since $\Delta^2-\Lambda \cdot\mathrm{Id}$ has elliptic principal part 
and constant coefficients, it is analytic hypoelliptic, see \cite[Chapter 3]{Treves}. 
Hence, $u$ is a real analytic function and by \cite[Proposition 0]{Mityagin}, 
its zero set has zero measure. The proof of this last statement relies on 
the Weierstrass preparation theorem. 
In conclusion,  $0\le A \le |\{u=0\}|=0$ 
and since $S$ is defined up to zero-measure sets, we can put $S=\{u=0\}$.    
\end{proof}

As a consequence of the previous result, we know in particular that $S$ contains a neighborhood of $\partial\Omega$.

The next proposition deals with the dependence of $\Lambda$ on the parameters $\alpha$ and $A$. This is the analogue of \cite[Proposition 10]{CGIKO00}. 
For notational ease, in what follows we write $S^c$ instead of $S^c \cap \Omega$.

\begin{proposition}\label{parameter_dep} The following properties hold
\begin{itemize} 
\item for $A>0$, $\Lambda(\alpha,A)$ is increasing in $\alpha$; 
\item $\Lambda(\alpha,A)$ is Lipschitz continuous in $\alpha$ with Lipschitz constant $A|\Omega|^{-1}$;
\item for $A< |\Omega|$, there exists a unique value of $\alpha$, denoted by $\bar \alpha(A)$,
such that $\Lambda(\alpha,A)=\alpha$;
\item for $A< |\Omega|$, $\Lambda(\alpha,A)-\alpha$ is decreasing in $\alpha$;
\item $\Lambda(\alpha,A)$ is continuous and non-decreasing in $A$. 
\end{itemize}
\end{proposition}
\begin{proof}
Let $A\in [0,|\Omega|]$ and take $0< \alpha_1<\alpha_2$ to fix the ideas. 
Denote $(u_1,S_1)$ and $(u_2,S_2)$ G-optimal pairs 
corresponding to $(\alpha_1,A)$ and $(\alpha_2, A)$ respectively. 
Without loss of generality, we can assume $\|u_1\|_{L^2(\Omega)}=\|u_2\|_{L^2(\Omega)}=1$. Then, 
by the optimality of $(u_1, S_1)$ for the data $(\alpha_1,A)$, we get 
$$
\begin{aligned}
\Lambda(\alpha_1,A)&=\int_\Omega(\Delta u_1)^2 +\alpha_1\int_{S_1}u_1^2\le \int_\Omega(\Delta u_2)^2+\alpha_1\int_{S_2}u_2^2\\
&\le \int_\Omega(\Delta u_2)^2 +\alpha_2\int_{S_2}u_2^2=\Lambda(\alpha_2,A)
\end{aligned}
$$
where the last inequality is strict if $A>0$, since $u_2$ cannot be zero a.e. in $S_2$. 
Indeed, if by contradiction $u_2=0$ a.e. in $S_2$, since $S_2$ 
is of the form $\{u_2^2\le t\}$ for some $t\ge0$, it results that $t=0$. 
By the discussion in the proof of Theorem \ref{Theo1}-(b), this implies 
that $A=0$, which is a contradiction. Hence, if $A>0$, 
$\Lambda(\alpha,A)$ is increasing in $\alpha$. 
On the other hand, by the optimality of $(u_2, S_2)$ for the data $(\alpha_2,A)$, we obtain
$$ 
\begin{aligned}
\Lambda(\alpha_2,A)&=\int_\Omega(\Delta u_2)^2 +\alpha_2\int_{S_2}u_2^2 \le \int_\Omega(\Delta u_1)^2 +\alpha_2\int_{S_1}u_1^2\\
&= \Lambda(\alpha_1,A)+(\alpha_2-\alpha_1)\int_{S_1}u_1^2 \le\Lambda(\alpha_1,A)+(\alpha_2-\alpha_1)\frac{A}{|\Omega|},
\end{aligned}
$$
where the last estimate comes from 
$$\frac{\int_{\{u^2\le t\}}
u^2}{|\{u^2\le t\}|}\le \frac{\int_\Omega u^2}{|\Omega|},$$
which in turn is a simple consequence of $\{u^2\le t\}\cup\{u^2>t\}=\Omega$, $\{u^2\le t\}\cap\{u^2>t\}=\emptyset$, and 
$$\fint_{\{u^2>t\}}u^2 \ge\fint_{\{u^2\le t\}}u^2.$$
Altogether, we get for $\alpha_1<\alpha_2$
$$0\le \Lambda(\alpha_2,A)-\Lambda(\alpha_1,A)\le (\alpha_2-\alpha_1)\frac{A}{|\Omega|}.$$
Analogously, if $\alpha_1>\alpha_2$ we have 
$$0\le \Lambda(\alpha_1,A)-\Lambda(\alpha_2,A)\le (\alpha_1-\alpha_2)\frac{A}{|\Omega|},$$
and so for all $\alpha_1,\,\alpha_2>0$
$$|\Lambda(\alpha_1,A)-\Lambda(\alpha_2,A)|\le \frac{A}{|\Omega|}|\alpha_1-\alpha_2|,$$
that is $\Lambda(\cdot,A)$ is Lipschitz continuous 
with Lipschitz constant $A|\Omega|^{-1}$. In particular, for $A<|\Omega|$, 
$\Lambda(\cdot,A)$ is a contraction mapping and, 
by the Banach fixed-point Theorem, it admits a unique fixed-point $\bar\alpha(A)$. 

Now, suppose that $A<|\Omega|$ and $0< \alpha_1<\alpha_2$, and estimate in the same notation as above
$$
\begin{aligned}
\Lambda(\alpha_2,A)-\alpha_2&\le \int_\Omega(\Delta u_1)^2+\alpha_2\int_{S_1}u_1^2-\alpha_2\\
&=\Lambda(\alpha_1,A)-\alpha_1-(\alpha_2-\alpha_1)\left(\int_\Omega u_1^2 -\int_{S_1}u_1^2\right).
\end{aligned}
$$
In order to prove that $\Lambda(\alpha,A)-\alpha$ 
is decreasing in $\alpha$ it remains to show that 
$$\int_\Omega u_1^2 -\int_{S_1}u_1^2>0.$$
We argue by contradiction and suppose 
$$\int_\Omega u_1^2 -\int_{S_1}u_1^2=\int_{S_1^c}u_1=0,$$
that is $u_1=0$ a.e in $S_1^c$. Since $S_1^c=\{u_1^2>t\}$ 
and $u_1$ is continuous, $S_1^c$ is open and, up to a 
translation, we can assume that $0\in S_1^c$. 
Furthermore, $|S_1^c|=|\Omega|-A>0$ 
and $|\Delta^2u_1|= |\Lambda(\alpha_1,A)\cdot\mathrm{Id}-\alpha_1\chi_{S_1}|\cdot|u_1|\le (\Lambda(\alpha_1,A)+\alpha_1)|u_1|$. Hence, by the Unique Continuation Theorem 
in \cite{Protter}, $u_1\equiv 0$ in $\Omega$. 
This is impossible being $\|u_1\|_{L^2(\Omega)}=1$ and concludes the proof of this part.

Finally, let $0\le A_1<A_2\le |\Omega|$ and $\alpha>0$. 
Denote $(u_1,S_1)$ and $(u_2,S_2)$ G-optimal pairs corresponding 
to the data $(\alpha,A_1)$ and $(\alpha,A_2)$ respectively. 
Let $S_2'\subset\Omega$ be such that $|S_2'|=A_2$ and $S_2'\supset S_1$. 
Then, by the optimality of $\Lambda(\alpha,A_2)$ we get
$$
\begin{aligned}
\Lambda(\alpha,A_1)&=\int_\Omega(\Delta u_1)^2+\alpha\int_{S_2'}u_1^2-\alpha\int_{S_2'\setminus S_1}u_1^2\\
&\ge \int_\Omega(\Delta u_2)^2+\alpha\int_{S_2}u_2^2-\alpha\int_{S_2'\setminus S_1}u_1^2=\Lambda(\alpha,A_2)-\alpha\int_{S_2'\setminus S_1}u_1^2.
\end{aligned}
$$ 
On the other hand, denoting by $S_1'$ a subset of $S_2$ 
having $|S_1'|=A_1$ and using the optimality of $\Lambda(\alpha,A_1)$, we have
$$
\Lambda(\alpha,A_1)\le\int_\Omega(\Delta u_2)^2+\alpha\int_{S_1'}u_2^2\le\int_\Omega(\Delta u_2)^2+\alpha\int_{S_2}u_2^2=\Lambda(\alpha,A_2).
$$ 
Therefore,
$$0\le \Lambda(\alpha,A_2)-\Lambda(\alpha,A_1)\le \alpha\int_{S_2'\setminus S_1}u_1^2,$$
and so $\Lambda(\alpha,\cdot)$ is non-decreasing and 
$$|\Lambda(\alpha,A_1)-\Lambda(\alpha,A_2)|\le \alpha\int_{S_2'\setminus S_1}u_1^2\;\to0 \quad\mbox{as }A_1\to A_2.$$
\end{proof}

\begin{proposition}\label{levelset}
Every set $\{u^2=s\}$, $s\ge0$, has zero measure, 
except possibly $\{u^2=t\}$ when $\alpha=\bar\alpha(A)$.
\end{proposition}
\begin{proof} 
We use the same notation as in the proof of Theorem \ref{Theo1}-(b). 
The argument is similar to the one contained in \cite[Theorem 1-(c)]{CGIKO00}, but we present
it here for the sake of completeness. If $s>t$, $\mathcal N_s\subset S^c$. Hence, 
$$0=\Delta^2 u=(\Lambda \cdot \mathrm{Id}-\alpha\chi_S)u= \Lambda u\quad\mbox{a.e. on }\mathcal N_s.$$
Since $\Lambda >0$ and $u\neq 0$ on $\mathcal N_s$, $|\mathcal N_s|=0$. 
Now, if $s=t$, $\mathcal N_s\subset S$ and so 
$$0=(\Lambda-\alpha) u\quad\mbox{a.e. on }\mathcal N_s.$$
Thus, if $\alpha\neq\bar\alpha(A)=\Lambda$, 
we can conclude again that $|\mathcal N_{t}|=0$. 
Finally, if $s<t$, again $\mathcal N_s\subset S$ and $\Delta^2u=(\Lambda-\alpha)u$ 
in the open set $\{u^2<t\}$. The function $v:=u-s$ solves the equation 
$$\Delta^2v=(\Lambda-\alpha)v+(\Lambda-\alpha) s\quad\mbox{in }\{u^2<t\}.$$
Therefore, $v$ is a real analytic function and so $|\{v=0\}|=|\{u=s\}|=0$.
\end{proof}

\section{Proof of Theorem \ref{Struct}}\label{Relation}
In this section, as in \cite{CGIKO00}, we highlight the relations between the two problems \eqref{Lambda} and \eqref{CP}, which  will be useful in proving the symmetry results later on.

\begin{proof}[$\bullet$ Proof of Theorem \ref{Struct}]
For $(a)$, let us consider a CP-minimizer $(u, \rho)$. 
We write any $\rho\in\mathrm{P}$ as $\rho = H + (\rho -H)$ and so the PDE in \eqref{4EL} (or \eqref{4ELd}) reads as
\begin{equation}\label{ELP2}
\Delta^2 u + \Theta (H-\rho) u = \Theta H u \quad \textrm{in } \Omega.
\end{equation}
\smallskip

\underline{\it Claim}: it is possible to choose $\alpha>0$ 
and $A\in[0,|\Omega|]$ for which \eqref{ELP2} can be seen in the form
$$
\Delta^2 u+\alpha\eta u=\Lambda(\alpha,A) u\quad\mbox{in }\Omega
$$
for some $\eta\in\mathcal A:=\{\eta:\Omega\to\mathbb R\,:\,0 \leq \eta \leq 1, \,\int_{\Omega}\eta= A\}$.
\medskip

In order to prove the claim, we put 
\begin{equation}\label{choice}
\alpha:=\Theta(H-h)>0,\quad\eta:=\frac{H-\rho}{H-h},\quad\mbox{and consequently}\quad A:=\frac{H |\Omega|-M}{H-h}.
\end{equation}
Thus, we need to show that 
\begin{enumerate}
\item[(i)] $0\le\frac{H-\rho}{H-h}\le1$;\\
\item[(ii)] $0\le\frac{H \, |\Omega|-M}{H-h}\le |\Omega|$;\\
\item[(iii)] $\Lambda\left(\Omega,\Theta(H-h),\frac{H\, |\Omega|-M}{H-h}\right)=\Theta H$.
\end{enumerate}
Now, (i) follows immediately by the bounds $h\le\rho\le H$, 
while (ii) follows from the assumption $M\in [h |\Omega|,H |\Omega|]$.
Hence, Proposition \ref{min-pb} applies and we know that 
$$
\Lambda(\alpha,A)=\inf_{\eta\in\mathcal A}\inf_{u\in \mathcal{W}\setminus\{0\}}\frac{\int_\Omega(\Delta u)^2+\alpha\int_\Omega\eta u^2}{\int_\Omega u^2},
$$
with $\alpha,\,\eta,\,A$ as in \eqref{choice}.
In terms of $\rho$, by \eqref{choice} this reads as
\begin{equation}\label{ThetaH}
\Lambda(\alpha,A)=\Theta H+\inf_{\rho\in\mathrm{P}}\inf_{u\in \mathcal{W}\setminus\{0\}}\frac{\int_\Omega(\Delta u)^2-\Theta\int_\Omega\rho u^2}{\int_\Omega u^2}.
\end{equation}
By the definition of $\Theta$ as an infimum, $\int_\Omega(\Delta u)^2-\Theta\int_\Omega\rho u^2\ge0$, hence 
\begin{equation}\label{ge0}
\inf_{\rho\in\mathrm{P}}\inf_{u\in \mathcal{W}\setminus\{0\}}\frac{\int_\Omega(\Delta u)^2-\Theta\int_\Omega\rho u^2}{\int_\Omega u^2}\ge0.
\end{equation}
On the other hand, since $\rho\le H$, and using again the definition of $\Theta$, we get 
\begin{equation}\label{le0}
\begin{aligned}
\inf_{\rho\in\mathrm{P}}\inf_{u\in \mathcal{W}\setminus\{0\}}\frac{\int_\Omega(\Delta u)^2-\Theta\int_\Omega\rho u^2}{\int_\Omega u^2}&= \inf_{\rho\in\mathrm{P}}\inf_{u\in \mathcal{W}\setminus\{0\}}\left(\frac{\int_\Omega(\Delta u)^2}{\int_\Omega\rho u^2}-\Theta\right)\frac{\int_\Omega\rho u^2}{\int_\Omega u^2}\\
&\le H\inf_{\rho\in\mathrm{P}}\inf_{u\in \mathcal{W}\setminus\{0\}}\left(\frac{\int_\Omega(\Delta u)^2}{\int_\Omega\rho u^2}-\Theta\right)=0.
\end{aligned}
\end{equation}
Combining together \eqref{ge0} and \eqref{le0}, we obtain
$$
\inf_{\rho\in\mathrm{P}}\inf_{u\in \mathcal{W}\setminus\{0\}}\frac{\int_\Omega(\Delta u)^2-\Theta\int_\Omega\rho u^2}{\int_\Omega u^2}=0,
$$
and, by \eqref{ThetaH}, (iii) is proved. This concludes the proof of the claim.

Now, by Proposition \ref{min-pb}, we know that $\eta = \chi_S$
for a set $S\in\mathcal S_t$ as in \eqref{SetS}. Hence, by \eqref{choice},
$$
\rho=H-(H-h)\eta=H-(H-h)\chi_S=h\chi_S+H\chi_{S^c},
$$ 
which closes the proof of part $(a)$. 

We are now ready to prove $(b)$. Here and in what follows, 
$\rho$ and $S$ are as in the statement of part $(a)$. 
We first observe that the ``only if'' part and the fact 
that $\Lambda(\alpha,A) = \Theta(h,H,M) H$ for 
$\alpha$, $A$ as in \eqref{Alpha}-\eqref{A} have been shown in the proof of $(a)$. 
Hence, it remains to prove that if $(u,S)$ realizes $\Lambda:=\Lambda(\alpha,A)$, 
with $\alpha$ as in \eqref{Alpha}, and $A$ as in \eqref{A}, 
then $(u,\rho)$ realizes $\Theta:=\Theta(h,H,M)$.
By assumption, we have
\begin{equation*}
\Lambda=\Theta H=\frac{\int_\Omega(\Delta u)^2+\Theta(H-h)\int_\Omega\chi_S u^2}{\int_\Omega u^2},
\end{equation*}
thus
\begin{equation*}
\Theta=\frac{\int_\Omega(\Delta u)^2-\Theta\int_\Omega\rho u^2+\Theta H\int_\Omega(\chi_S+\chi_{S^c})u^2}{H\int_\Omega u^2}=\frac{\int_\Omega(\Delta u)^2-\Theta\int_\Omega\rho u^2}{H\int_\Omega u^2}+\Theta. 
\end{equation*}
Therefore, 
$$
\Theta\int_\Omega\rho u^2=\int_\Omega(\Delta u)^2,
$$
that is $(u,\rho)$ realizes $\Theta$.

For part (c), we observe that, by $h<H$, $\Lambda>0$ 
and \eqref{Alpha}, we immediately get $\alpha>0$. Furthermore, if $A\in[0,|\Omega|)$, 
$$
\alpha=\frac{H-h}{H}\Lambda(\alpha,A)\begin{cases}
<\Lambda(\alpha,A)\quad	&\mbox{if }h>0,\\
=\Lambda(\alpha,A)\quad	&\mbox{if }h=0.
\end{cases}
$$
While, if $A=|\Omega|$, $S=\Omega$ and $\rho\equiv h$ by part (a). 
Thus, $\Theta(H,h,M)=\mu(\Omega)/h$ for any $H>h$. Hence, by \eqref{Alpha}, 
$\alpha=(H-h)\mu(\Omega)/h$ can take any value in $(0,\infty)$ varying $H\in(h,\infty)$.
\end{proof}

\begin{remark}\label{case_h=0}
We end this section by noting explicitly that, by part (a) of the previous theorem it follows in particular that, if $h=0$, $\rho\equiv 0$ in $S$. Now, since $|\{\rho=0\}|=0$ by the definition of admissible densities $\rho\in\mathrm{P}$, $|S|=0$. Moreover, since $S$ is defined up to a zero-measure set, $S^c=\Omega$. Therefore, when $h=0$ problem \eqref{CP} reduces to the standard eigenvalue problem for the biharmonic operator.\\
We further observe that by the very definition of $t=t(u)$ in \eqref{SetS}, denoting by $(u, \rho_u)$ a CP-optimal pair,
since $\rho_u = h \chi_{\{u^{2}\leq t(u)\}} + H \chi_{\{u^2 > t(u)\}}$, we have that $\rho_{\mu u} = \rho_u$. Indeed
$t(\mu u)= \mu^2 t(u)$ and so $\{u^{2}\leq t(u)\} = \{(\mu u)^{2}\leq t(\mu u)\}$. 
\end{remark}

\section{Proof of Theorem \ref{Theo-ball}}\label{Sec5} 
The aim of this section is to address qualitative properties of the CP-optimal pairs 
$(u,\rho)$, such as positivity and radial symmetry 
in the case $\Omega=B:=\{x\in\mathbb R^n\,:\,|x|<1\}$. 

We start with the positivity of $u$.  

\begin{proposition}\label{positivity}
Let $\Omega=B$ and let $(u,\rho)$ be a CP-optimal pair, then $u>0$ in $B$.
\end{proposition}
\begin{proof}
Let $w$ be a solution of 
\begin{equation}\label{Pw}
\Delta^2 w=\Theta \rho |u|\quad\mbox{in }B,
\end{equation}
coupled either with Navier or with Dirichlet boundary conditions.
By Lemma \ref{max_principle}, $w>0$ a.e. in $B$, otherwise 
we would have $u\equiv 0$ in $B$ which is impossible. Now, suppose by 
contradiction that $u$ is sign-changing and consider the functions $w-u$ and $w+u$. Then
$$
\Delta^2(w-u)=2\Theta \rho u^-\quad\mbox{and}\quad\Delta^2(w+u)=2\Theta \rho u^+\quad\mbox{in }B.
$$ 
Hence, 
$$
\int_B\Delta(w-u)\Delta v dx\ge0\quad\mbox{and}\quad\int_B\Delta(w+u)\Delta v dx\ge0\quad\mbox{for all }v\in\mathcal C^+.
$$ 
Again by Lemma \ref{max_principle}, we get that either 
$\pm u\equiv w$ or $|u|<w$ a.e. in $B$. In the first case, 
being $w>0$, up to a change of sign of $u$, we are done. 
In the latter, we multiply \eqref{Pw} by $w$, integrate over $B$ and get 
$$
\int_B(\Delta w)^2=\Theta\int_B \rho |u|w<\Theta\int_B \rho w^2,
$$
which implies
$$
\frac{\int_B(\Delta w)^2}{\int_B \rho w^2}<\Theta.
$$ 
This contradicts the minimality of $\Theta$ and concludes the proof.
\end{proof}

\begin{remark}
As for Lemma \ref{max_principle}, if we deal with 
Navier boundary conditions, we can consider more
general open sets in Proposition \ref{positivity}.\\  
There is a simple consequence of the positivity
result in Proposition \ref{positivity}: for $\alpha\le\bar\alpha(A)$
we have an equivalence between \eqref{CP} and \eqref{Lambda} therefore, recalling
Theorem~\ref{Struct}, the optimal set $S$ can be written as 
a sub-level set of the function $u$ itself, i.e.
$$S=\{u\le \sqrt{t}\}.$$
\end{remark}
\medskip

For the symmetry issues we need to distinguish the case with Dirichlet boundary conditions from the one with Navier boundary conditions. 

Before proving the symmetry result for {\it Dirichlet boundary conditions}, we need to prove some 
preliminary lemmas. 

In the rest of the section we consider a CP-optimal pair $(u,\rho)$ and we extend  
$u\in C_0(\overline{B}):=\{\varphi\in C(\overline{B})\,:\, \varphi=0\mbox{ on }\partial B\}$ 
by defining it to be zero outside 
$B$. We must consider an extension of $\rho$ as well.
We will denote it by 
$$\rho_u:=h\chi_{\{u\le\sqrt{t}\}}+H\chi_{\{u>\sqrt{t}\}},$$
\noindent where we are considering sub-level sets of the extended
function $u$.

\begin{lemma}\label{Lemma-fufuh}
Let $\mathcal H\subset\mathbb R^n$ be a half-space. Then
$$
[\rho_u \, u]_{\mathcal H}\equiv\rho_{u_\mathcal H} \, u_{\mathcal H}.
$$
\end{lemma}
\begin{proof}
We prove this lemma by using the definitions of the two functions involved, namely 
\begin{equation*}
[\rho_{u} \, u]_{\mathcal H}(x)=\left\{ \begin{array}{rl}
\max\{\rho_{u}(x) u(x),\rho_{u}(\bar x)u(\bar x)\}, & \quad \mbox{if }x\in\mathcal H,\\
\min\{\rho_{u}(x) u(x),\rho_{u}(\bar x)u(\bar x)\}, & \quad \mbox{if }x\in\mathbb R^n\setminus\mathcal H,
\end{array}\right.
\end{equation*}
and
$$
\rho_{u_\mathcal H}(x)u_{\mathcal H}(x)=\begin{cases}
hu_\mathcal H(x),\quad&\mbox{if }u_\mathcal H(x)\le \sqrt{t},\\
Hu_\mathcal H(x),\quad&\mbox{if }u_\mathcal H(x)> \sqrt{t}.
\end{cases}
$$
Now, for every $x\in\mathbb R^n$ four cases may occur: 
\begin{itemize}
\item $x\in\{u\le \sqrt{t}\}$ and $\bar x\in \{u\le \sqrt{t}\}$;
\item $x\in\{u\le \sqrt{t}\}$ and $\bar x\not\in \{u\le \sqrt{t}\}$;
\item $x\not\in\{u\le \sqrt{t}\}$ and $\bar x\in \{u\le \sqrt{t}\}$;
\item $x\not\in\{u\le \sqrt{t}\}$ and $\bar x\not\in \{u\le \sqrt{t}\}$.
\end{itemize}
We start with considering $x\in\mathcal H$. 
In the first case
$$[\rho_u \, u]_{\mathcal H}(x)=\max\{h u(x),h u(\bar x)\}=hu_\mathcal H(x).$$
Furthermore, since $u(x)\le\sqrt{t}$ and $u(\bar x)\le\sqrt{t}$, 
also $u_\mathcal H(x)=\max\{u(x),u(\bar x)\}\le\sqrt{t}$ and so  
$$\rho_{u_\mathcal H}(x)u_{\mathcal H}(x)=hu_\mathcal H(x).$$
If the second case occurs, we know that $u(\bar x)>u(x)$ and consequently
$$[\rho_u\, u]_{\mathcal H}(x)=\max\{h u(x),H u(\bar x)\}=H u(\bar x)=Hu_\mathcal H(x).$$
On the other hand, since $u(\bar x)>\sqrt{t}$, also 
$u_\mathcal H(x)=\max\{u(x),u(\bar x)\}>\sqrt{t}$, which implies 
$$\rho_{u_\mathcal H}(x)u_{\mathcal H}(x)=Hu_\mathcal H(x)$$
and concludes the proof also in this case. With similar arguments 
it is possible to check the remaining cases both for 
$x\in\mathcal H$ and $x\in\mathbb R^n\setminus\mathcal H$. 
\end{proof}

\begin{lemma}\label{Lemma-uuh}
Let $(u,\rho_{u})$ be a CP-optimal pair in the ball $B$ 
and $\mathcal H\subset\mathbb R^n$ a half-space. If 
$\rho_u  u=[\rho_u u]_{\mathcal H}$,
then $u=u_\mathcal H$. 
\end{lemma}
\begin{proof} 
Suppose first that $h>0$. By hypothesis and Lemma \ref{Lemma-fufuh}, we know that for every $x\in\mathbb R^n$
\begin{equation}\label{equiv}
\begin{cases}
hu(x)\quad&\mbox{if }u(x)\le\sqrt{t}\\
Hu(x)&\mbox{if }u(x)>\sqrt{t}
\end{cases}\quad= \quad
\begin{cases}
hu_\mathcal H(x)\quad&\mbox{if }u_\mathcal H(x)\le\sqrt{t}\\
Hu_\mathcal H(x)&\mbox{if }u_\mathcal H(x)>\sqrt{t}.
\end{cases}
\end{equation}
Suppose by contradiction that there exists $x\in\mathbb R^n$ 
such that $u_\mathcal H(x)\neq u(x)$. Then, if $u(x)\le \sqrt{t}$ 
and $u_\mathcal H(x)\le \sqrt{t}$, by \eqref{equiv}, 
$hu(x)=hu_\mathcal H(x)$ which is absurd. Analogously, 
the case $u(x)> \sqrt{t}$ and $u_\mathcal H(x)> \sqrt{t}$ 
cannot occur if $u_\mathcal H(x)\neq u(x)$. Now, 
if $u(x)\le \sqrt{t}$ and $u_\mathcal H(x)> \sqrt{t}$, 
by \eqref{equiv}, we get $hu(x)=Hu_\mathcal H(x)$ and 
so clearly $u(x)\neq 0$. Hence, 
$$
\sqrt{t}<u_\mathcal H(x)=\frac{h}{H}u(x)<u(x)\le \sqrt{t}
$$ 
which is a contradiction. Analogously, we can rule out 
the opposite case $u(x)> \sqrt{t}$ and $u_\mathcal H(x)\le \sqrt{t}$, 
and conclude the proof for $h>0$. 

If $h=0$, by Remark \ref{case_h=0}, $u>\sqrt{t}$ in $B$, and consequently $u_\mathcal H$ can attain values only in $\{0\}\cup (\sqrt{t},\infty)$. Now, if $x\in\mathbb R^n\setminus B$, $u(x)=0$. Then,  in view of \eqref{equiv}, for every $x\in\mathbb R^n\setminus B$
$$
0=\left\{\begin{array}{ll}
0,&\quad\mbox{if }u_\mathcal H(x)\le \sqrt{t},\\
Hu_\mathcal H(x),&\quad\mbox{if }u_\mathcal H(x)> \sqrt{t},
\end{array}
\right.
$$
which implies that $u_\mathcal H(x)\le \sqrt{t}$ and therefore $u_\mathcal H(x)=0$. Hence, $u(x)=0= u_\mathcal H(x) $ for every $x\in \mathbb R^n\setminus B$. Analogously it can be seen that $u\equiv u_\mathcal H$ in $B$ and the proof is concluded.
\end{proof}

Let $G: B\times B\to\mathbb R$ be the Green function for the biharmonic operator with Dirichlet boundary conditions on the ball. We recall that $G$ has an explicit representation due to Boggio \cite{Boggio}. We are now considering the trivial zero extension of $G$ to the whole of $\mathbb R^n\times \mathbb R^n$.  
We define $\tilde u:\mathbb R^n\to\mathbb R$ as 
$$
\tilde u(x):=\Theta\int_{\mathbb R^n}G(x,y)\rho_u (y)u(y)\, dy,
$$
then $\tilde u\equiv 0$ in $\mathbb R^n\setminus B$ 
and $\tilde u|_{B}$ is the unique solution of the problem 
\begin{equation*}
\left\{\begin{array}{rl}
\Delta^2 v=\Theta \left(h\chi_{\{u\le\sqrt{t}\}}+H\chi_{\{u>\sqrt{t}\}}\right)u, & \mbox{in }B,\medskip\\

v=\frac{\partial v}{\partial \nu}=0, & \mbox{on }\partial B.
\end{array}\right.
\end{equation*}
By uniqueness and by the trivial extension of $u$, 
if $(u,\rho)$ is a CP-optimal pair, $\tilde u\equiv u$.

\begin{lemma}\label{Lemma-w}
Let $\mathcal{H}$ be a half-space such that 
$0\in\mathrm{int}(\mathcal H)$, and for every $x\in\mathbb R^n$
\begin{equation}\label{def-w}
w(x):=\Theta\int_{\mathbb R^n} G(x,y)\rho_{u_{\mathcal H}}(y)u_{\mathcal H}(y)dx.
\end{equation}
Then the following inequalities hold
\begin{itemize}
\item[(i)] $w(x)\ge w(\bar x)$ for every $x\in\mathcal H$;
\item[(ii)] $w(x)\ge u_{\mathcal H}(x)$ for every $x\in\mathcal H$;
\item[(iii)] $w(x)+w(\bar x)\ge u_{\mathcal H}(x)+ u_{\mathcal H}(\bar x)$ for every $x\in\mathbb R^n$.
\end{itemize}
Moreover, if $\rho_u u \not\equiv[\rho_u u]_{\mathcal H}$, 
then (iii) is strict for every $x\in\mathrm{int}(B\cap\mathcal H)$.
\end{lemma}
\begin{proof}
For the proofs of (i), (ii), and (iii), we refer to \cite[Lemma 4]{FGW}. 
We now show the last part of the statement whose proof is slightly 
different from the one contained in \cite{FGW}, since in our case 
the function $f$ is $\rho_u u$ which is not continuous. However, formula (4.18) of \cite{FGW} still holds, namely for every $x\in\mathbb R^n$
\begin{equation}\label{ineq}
\left.\begin{array}{l}
w(x)+w(\bar x)-[u_\mathcal H(x)+u_\mathcal H(\bar x)]\\
=\Theta\displaystyle{\int_\mathcal H}\Big(G(x,y)+G(\bar x,y)-[G(x,\bar y)+G(\bar x,\bar y)]\Big)\Big(\rho_{u_\mathcal H}(y)u_\mathcal H(y)-\rho_u (y)u(y)\Big)dy \ge0. 
\end{array}\right.
\end{equation}
By Lemma \ref{Lemma-G}, we know that if $x,\,y\in\mathrm{int}(B\cap\mathcal H)$, 
$$G(x,y)+G(\bar x,y)>G(x,\bar y)+G(\bar x,\bar y),$$
thus, by \eqref{ineq} and by Lemma \ref{Lemma-fufuh}, 
it is enough to prove that we can find a positive-measure 
subset of $\mathrm{int}(B\cap\mathcal H)$ in which  
$$\rho_{u_\mathcal H} u_\mathcal H>\rho_u u.$$
We first observe that 
\begin{equation}\label{u=uH=0}
\rho_{u_\mathcal H} u_\mathcal H\equiv \rho_u u\equiv 0\quad\mbox{in }\mathbb R^n\setminus B.
\end{equation}
Indeed, since $u>0$ in $B$ and $u\equiv 0$ in $\mathbb R^n\setminus B$, 
\begin{equation}\label{uHxB}
u_\mathcal H(x)=\begin{cases}
\max\{0,u(\bar x)\}=u(\bar x),\quad&\mbox{if }x\in\mathcal H,\\
\min\{0,u(\bar x)\}=0,&\mbox{if }x\in\mathbb R^n\setminus\mathcal H
\end{cases}
\end{equation}
for every $x\in\mathbb R^n\setminus B$.
Furthermore, since $0\in\mathrm{int}(\mathcal H)$, 
$|\bar x|\ge |x|$ for every $x\in\mathcal H$.
Thus, $x\not\in B$ implies $\bar x\not\in B$, 
and so $u(\bar x)=0$ in the first line of the 
definition \eqref{uHxB}, which yields \eqref{u=uH=0}.  Moreover, $u_\mathcal H\equiv u$ on $B\cap \partial\mathcal H$, because for every $x\in\partial \mathcal H$ it holds $x=\bar x$.
Therefore, $\rho_{u_\mathcal H} u_\mathcal H\not\equiv\rho_u u$ 
ensures that there exists $y\in B\setminus\partial \mathcal H$ 
for which $\rho_{u_\mathcal H}(y)u_\mathcal H(y)\neq\rho_u (y)u(y)$. 
We can always assume $y\in\mathrm{int}(B\cap\mathcal H)$, 
since if this is not the case, $\bar y$ will do the job, 
being by \eqref{identity} and Lemma \ref{Lemma-fufuh}
$$0\neq\rho_{u_\mathcal H}(y)u_\mathcal H(y)-\rho_u (y)u(y)=\rho_u (\bar y)u(\bar y)-\rho_{u_\mathcal H}(\bar y)u_\mathcal H(\bar y).$$
Hence, there exists $y\in\mathrm{int}(B\cap\mathcal H)$ such that 
\begin{equation}\label{neq}
\begin{cases}
hu(y)\quad&\mbox{if }u(y)\le\sqrt{t}\\
Hu(y)&\mbox{if }u(y)>\sqrt{t}
\end{cases}\quad\neq\quad
\begin{cases}
hu_\mathcal H(y)\quad&\mbox{if }u(y)\le\sqrt{t}\\
Hu_\mathcal H(y)&\mbox{if }u(y)>\sqrt{t}.
\end{cases}
\end{equation}
Now, since $y\in\mathcal H$, if $u(y)>\sqrt{t}$, also $u_\mathcal H(y)>\sqrt{t}$, 
and so we have only the following three possible cases.
\smallskip

\underline{\it Case $u(y)>\sqrt{t}$.} By \eqref{neq} 
and the fact that $y\in\mathcal H$, we know that $Hu(y)<Hu_\mathcal H(y)$. 
Hence, by the continuity of $u$ and $u_\mathcal H$ 
we can find a neighborhood $U_y$ of $y$ such that 
$$U_y\subset\mathrm{int}(B\cap \mathcal H)\cap\{u>\sqrt{t}\}\cap\{u_\mathcal H>\sqrt{t}\}$$ and 
\begin{equation}\label{conclusion}
\rho_{u_\mathcal H}(x)u_\mathcal H(x)>\rho_u (x)u(x)\quad\mbox{for every }x\in U_y.
\end{equation}
\smallskip

\underline{\it Case $u(y)\le\sqrt{t}$ and $u_\mathcal H(y)>\sqrt{t}$.} 
Again, by \eqref{neq}, we get 
$Hu_\mathcal H(y)\neq hu(y)$, and since $u_\mathcal H\ge u$ 
in $\mathcal H$ and $H>h$, this yields 
$$Hu_\mathcal H(y)\ge Hu(y)>hu(y).$$ Now, if $u(y)<\sqrt{t}$, 
we can find a neighborhood $U_y$ such that 
$$U_y\subset\mathrm{int}(B\cap \mathcal H)\cap\{u<\sqrt{t}\}\cap\{u_\mathcal H>\sqrt{t}\}$$ 
and $Hu_\mathcal H(x)>hu(x)$ for every $x\in U_y$, 
that is to say \eqref{conclusion} holds also in this case.
If $u(y)=\sqrt{t}$ then clearly $u_\mathcal H(y)>u(y)$ 
and by continuity there exists a neighborhood 
$U_y\subset\mathrm{int}(B\cap\mathcal H)\cap\{u_\mathcal H >\sqrt{t}\}$ where  
$u_\mathcal H>u$. This implies that 
$$Hu_\mathcal H(x)>Hu(x)>hu(x)\quad\mbox{for every }x\in U_y,$$
and in turn \eqref{conclusion} holds for both  $x\in\{u\le\sqrt{t}\}$ and $x\in\{u>\sqrt{t}\}$. 
\smallskip

\underline{\it Case $u(y)\le\sqrt{t}$ and $u_\mathcal H(y)\le\sqrt{t}$.} 
By \eqref{neq}, we get $hu_\mathcal H(y)>hu(y)$ and by continuity 
we can find $U_y\subset\mathrm{int}(B\cap\mathcal H)$ where 
$u_\mathcal H>u$. Let $x\in U_y$. If $u_\mathcal H(x)\le\sqrt{t}$, 
then also $u(x)\le \sqrt{t}$, and so $hu_\mathcal H(x)>hu(x)$ is 
equivalent to $\rho_{u_\mathcal H}(x)u_\mathcal H(x)>\rho_u (x)u(x)$. 
If $u_\mathcal H(x)>\sqrt{t}$, then 
$$H u_\mathcal H(x)>Hu(x)>hu(x).$$
Hence, for both  $x\in\{u\le\sqrt{t}\}$ and $x\in\{u>\sqrt{t}\}$, 
$$\rho_{u_\mathcal H}(x)u_\mathcal H(x)>\rho_u (x)u(x).$$ Then, 
also in this case \eqref{conclusion} holds, which concludes the proof. 
\end{proof}

\begin{lemma}\label{Lemma-=}
Let $\mathcal H\subset \mathbb R^n$ be a half-space with 
$0\in\mathrm{int}(\mathcal H)$, and $w$ be defined as in \eqref{def-w}. Then,
\begin{equation}\label{inequality=}
\int_B w\rho_{u_\mathcal H} u_\mathcal H \le \int_B \rho_{u_\mathcal H} w^2.
\end{equation}
Furthermore, if equality holds, then $\rho_u u\equiv[\rho_u u]_{\mathcal H}$. 
\end{lemma}
\begin{proof}
By Lemma \ref{Lemma-w} we get 
\begin{equation}\label{miracle}
\begin{aligned}
\int_B[\rho_{u_\mathcal H} w^2-\rho_{u_\mathcal H} u_\mathcal H w]dx&=\int_\mathcal H\{\rho_{u_\mathcal H}(x)w(x)[w(x)-u_\mathcal H(x)]\\
&\qquad\quad+\rho_{u_\mathcal H}(\bar x)w(\bar x)[w(\bar x)-u_\mathcal H(\bar x)]\}dx\\
&\ge\int_\mathcal H[w(x)-u_\mathcal H(x)]\cdot[\rho_{u_\mathcal H}(x)w(x)-\rho_{u_\mathcal H}(\bar x)w(\bar x)]dx\ge 0.
\end{aligned}
\end{equation}
We stress that in the last inequality we have also used the fact that, 
if $x\in\mathcal H$ then $\bar x\not\in\mathcal H$ and in particular
$$u_\mathcal H(x)=\max\{u(x),u(\bar x)\}\ge\min\{u(x),u(\bar x)\}=u_\mathcal H(\bar x).$$
Consequently, if $u_\mathcal H(x)\le\sqrt{t}$, then also $u_\mathcal H(\bar x)\le\sqrt{t}$, and so 
$$\rho_{u_\mathcal H}(x)\ge\rho_{u_\mathcal H}(\bar x)\quad\mbox{for every }x\in\mathcal H.$$
We can now prove the last part of the statement as in 
\cite[Lemma 5]{FGW}. If equality holds in \eqref{inequality=}, 
then also in \eqref{miracle} we have equality. This is 
only possible in two situations: $w(x)-u_\mathcal H(x)=u_\mathcal H(\bar x)-w(\bar x)$ 
for every $x\in\mathrm{int}(\mathcal H\cap B)$, 
or $\rho_{u_\mathcal H}(\bar x)w(\bar x)=0$ for 
all $x\in\mathrm{int}(\mathcal H\cap B)$. In the first case, 
we conclude by Lemma \ref{Lemma-w} that $\rho_u u \equiv[\rho_u u]_{\mathcal H}$. 
If the second case occurs, then since both $w$ and 
$\rho_{u_\mathcal H}$ are positive in $B$, we conclude 
that $B\subset\mathcal H$ and so again $\rho_u u \equiv[\rho_u u]_{\mathcal H}$, 
being $u\equiv 0$ outside $B$.
\end{proof}

We are now ready to end the proof of Theorem \ref{Theo-ball} for Dirichlet boundary conditions. 

\begin{proof}[$\bullet$ Proof of Theorem \ref{Theo-ball} for Dirichlet]
Let $(u,\rho_u)$ be a CP-optimal pair, with $u>0$ in $B$, 
and let $\mathcal H\subset \mathbb R^n$ be a half-space 
such that $0\in\mathrm{int}(\mathcal H)$. Then, by the definition 
\eqref{def-w} of $w$ 
we know that $w$ solves the problem 
$$
\left\{\begin{array}{rl}
\Delta^2 v=\Theta\rho_{u_\mathcal H} u_\mathcal H,&\quad\mbox{in }B,\smallskip\\
v=\frac{\partial v}{\partial \nu}=0,&\quad\mbox{on }\partial B.
\end{array}\right.
$$
Thus, by Lemma \ref{Lemma-=} we get
\begin{equation}\label{estimate}
\|\Delta w\|_{L^2(B)}^2=\Theta\int_B w\,\rho_{u_\mathcal H} u_\mathcal H\le \Theta\int_B \rho_{u_\mathcal H} w^2, 
\end{equation}
and so 
$$
\frac{\|\Delta w\|_{L^2(B)}^2}{\int_B \rho_{u_\mathcal H} w^2}\le \Theta.
$$
By Proposition \ref{equimisurabilita} also $\rho_{u_\mathcal{H}}$ 
is an admissible density (i.e. $\int_B\rho_{u_\mathcal H}=M$), 
then by the minimality of $\Theta$ equality must hold in \eqref{estimate}, 
and so $u=u_\mathcal H$ by Lemmas \ref{Lemma-=} and \ref{Lemma-uuh}. 
Therefore, by the arbitrariness of $\mathcal H$ and by Lemma \ref{Lemma-characterization}, we get that $u$ is a radial, 
radially non-increasing function and by its shape, $S$ is radial 
and $S^c$ is convex. In view of Proposition \ref{levelset}
and since $S$ is defined up to a set of measure zero, $S$ is the unique {\it open} shell region 
of measure $A$, $S=\{x\,:\, r(A)<|x|< 1\}$. In particular, $S$ and $S^c$ are of class $C^\infty$. 
In conclusion, for $\Omega = B$ there is a unique CP-optimal pair $(u,\rho)$.  It remains to prove the strict monotonicity of the radial profile of $u$. To this aim, we observe that, thanks to the regularity of the boundaries of $S$ and $S^c$ and the fact that $\rho$ is constant in both $S$ and $S^c$, $u\big|_S$ and $u\big|_{S^c}$ are of class $C^4$ in $\mathrm{int}(S)$ and $\mathrm{int}(S^c)$ respectively, cf. \cite[Theorem 2.20]{GGS}. Now, we have just proved that $u$ is radially non-increasing. Suppose by contradiction that there exists an open subset $U$ of $B$ where $u$ is constant, consequently either $U\subset S=\{u<\sqrt{t}\}$ or $U\subset S^c=\{u>\sqrt{t}\}$. Thus, $\Delta^2u=0$ in $U$, which contradicts the positivity of $u$, being $\Delta^2 u=\Theta\rho u>0$ in all of $B$. Here we are tacitly assuming $h>0$, the case $h=0$ being even simpler.
\end{proof}

We consider now the case of {\it Navier boundary conditions}. Here we can write our fourth-order problem \eqref{Nav} as the second-order system \eqref{NSys}, that is 
$$
\left\{ \begin{array}{rl}
          -\Delta u = v, & \quad \textrm{in $B$},\\
					-\Delta v = \Theta \, \rho \, u, & \quad \textrm{in $B$},\\
					u = v = 0, & \quad  \textrm{on $\partial B$}.
				\end{array}\right.					
$$

\begin{proposition}\label{radialNav}
Let $(u,v)$ be a weak solution of \eqref{Nav} such that $u>0$ and $v>0$ in $B$. Then $u$ and $v$ are radial and radially decreasing in $B$. 
\end{proposition}
\begin{proof}
For the proof of this result we refer to the ones of \cite[Theorem 1 and Lemmas~4.1-4.3]{Troy81} for system \eqref{Troy}. We just skecth the proof below, and we highlight how we can overcome the lack of the regularity assumptions required in \cite{Troy81} to the solution $(u_i)_{i=1}^m$ (i.e., $u_i\in C^2(\overline{B})$) and the nonlinearity $(f_i)_{i=1}^m$ (i.e., $f_i\in C^1$) of \eqref{Troy}, thanks to the special form of our system. 

As in \cite{Troy81}, we arbitrarily choose the $x_1$ axis and denote by $T_\xi$ the hyperplane $\mathrm{e}_1\cdot x=\xi$.
Since $B$ is bounded, for sufficiently large $\xi>0$, the plane $T_{\xi}$ does not intersect $\overline B$. We decrease $\xi$ (i.e., the plane $T_\xi$  moves continuously toward $B$,
preserving the normal) until $\xi_0$, that is the smallest value of $\xi$ for which $T_{\xi}$ begins to intersect $B$. From $\xi=\xi_0$ to $\xi=0$, the plane $T_\xi$, cuts off from $B$ an open set $\Sigma(\xi)$, which is the part of $B$ that does not contain the origin. Let $\Sigma'(\xi)$ denote the reflection of $\Sigma(\xi)$ with respect to the plane
$T_\xi$. For every $x\in \Sigma(\xi)$, we denote by $x^\xi$ the reflection of $x$ with respect to $T_\xi$.

The proof can be split into the following three steps. \smallskip

\underline{\it Step 1.}  {\it Let $x_0\in\partial B$ be such that $\nu^{(1)}(x_0)>0$. Then there exists $\delta>0$ such that $\frac{\partial u}{\partial x_1}<0$ and $\frac{\partial v}{\partial x_1}<0$ in $B\cap B(x_0,\delta)$.} 

This can be proved as in \cite[Lemma 4.1]{Troy81}. We observe that in our case $f_1(v)=v$ and $f_2(u)=\Theta\rho_u u$, hence $f_i(0)=0$ for $i=1,\,2$. This allows to avoid the case (ii) in the proof of \cite[Lemma 4.1]{Troy81} which would require the $C^2$-regularity of $v=\Delta u$.

Now, take $\xi\in (0,\xi_0)$ sufficiently close to $\xi_0$.
Since $\nu^{(1)}(x)>0$ for every $x\in\partial B\cap \partial(\Sigma(\xi))$, as a consequence of Step~1., it follows that for every $x\in\Sigma(\xi)$
\begin{equation}\label{<<}
\frac{\partial u}{\partial x_1}(x)<0, \quad \frac{\partial v}{\partial x_1}(x)<0, \quad u(x)<u(x^\xi),\quad v(x)<v(x^\xi).
\end{equation}

As in the proof of \cite[Lemma 4.3]{Troy81}, decrease $\xi$ below $\xi_0$ until a critical value $\bar\xi\ge0$ beyond which \eqref{<<} does not hold any more for $u$ or $v$. Then, for every $x\in\Sigma(\bar\xi)$
\begin{equation}\label{<=}
\frac{\partial u}{\partial x_1}(x)\le 0, \quad \frac{\partial v}{\partial x_1}(x)\le 0, \quad u(x)\le u(x^{\bar\xi}),\quad v(x)\le v(x^{\bar\xi}).
\end{equation}

\underline{\it Step 2.}  {\it Let $\xi\in(0,\xi_0)$, then 
$$
\begin{aligned}
&u(x)<u(x^{\xi}),\quad v(x)<v(x^\xi)\quad\mbox{for every }x\in\Sigma(\xi),\\  
&\frac{\partial u}{\partial x_1}(x)<0,\quad \frac{\partial v}{\partial x_1}(x)<0\quad\mbox{for every }x\in B\cap T_\xi.
\end{aligned}
$$}

This can be proved by using \eqref{<<} and \eqref{<=} as in \cite[Lemma 4.2]{Troy81}. We observe that the special form of $f_i$, $i=1,\,2$ (i.e., the fact that $f_1$ does not depend on $u$ and $f_2$ does not depend on $v$) allows us to avoid the use of the Mean Value Theorem in this proof.  Furthermore, the proof of \cite[Lemma 4.2]{Troy81} relies on the Hopf Lemma and the Strong Maximum Principle for $C^2$-solutions of second-order elliptic equations in domains with corners. In our case we can apply the Strong Maximum Principle 
and the Hopf Lemma in e.g. \cite[Theorem~2.2]{DP13} or \cite[Theorem~2.5.1, Theorem~2.7.1 and comments on p. 40]{PS},
which require only $C^1(\overline{B})$ regularity of the solution $(u,v)$.

As a consequence of Step 1. and Step 2., it is possible to prove that the value $\bar\xi\ge 0$ is indeed equal to $0$. This can be done by following the argument by contradiction proposed in \cite[Lemma 4.3]{Troy81}-Case (i). Here again the use of the Mean Value Theorem can be avoided thanks to the special form of the $f_i$'s in our problem. 
  
Furthermore, by Step 2., we get
$$
\frac{\partial u}{\partial x_1}(x)>0,\quad \frac{\partial v}{\partial x_1}(x)>0\quad\mbox{for every }x\in B\cap \{x\in\mathbb R^n\,:\,x_1<0\}
$$
and by continuity of partial derivatives of $u$ and $v$, 
\begin{equation}\label{==}
\frac{\partial u}{\partial x_1}(x)=0,\quad \frac{\partial v}{\partial x_1}(x)=0\quad\mbox{for every }x\in B\cap T_0.
\end{equation}

\underline{\it Step 3.}  {\it The functions $u$ and $v$ are symmetric with respect to the plane $T_0$.}

This can be proved as in \cite[Lemma 4.2]{Troy81}, by using \eqref{==}.
\smallskip 

The conclusion of the proof then follows by the arbitrariness of the $x_1$ axis. 
\end{proof}

\begin{proof}[$\bullet$ Proof of Theorem \ref{Theo-ball} for Navier]
By Proposition \ref{positivity} $u>0$, this together with the strong maximum principle implies that $v>0$. Therefore, we can apply Proposition \ref{radialNav}. The conclusion of Theorem \ref{Theo-ball} for Navier, concerning the properties of $S$, can be repeated verbatim as in the case for Dirichlet boundary conditions. 
\end{proof}
\begin{remark}
Let us denote by $\Theta_N$ and $\Theta_D$ the values of \eqref{CP} with Navier and Dirichlet boundary conditions, respectively. Since $H^2_0(\Omega)\subset H^2(\Omega)\cap H^1_0(\Omega)$, $\Theta_N\le \Theta_D$.   We can follow the argument in \cite{FGW} to prove that actually the strict inequality holds, namely
$$\Theta_N< \Theta_D.$$ 
Indeed, let $(u,\rho)\in H^2(\Omega)\cap H^1_0(\Omega)\times\mathrm{P}$ be a CP-optimal pair for Navier. Let us
assume by contradiction that $u$ does not have a sign in $\Omega$. Consider now the problem
\begin{equation}\label{COMP}
\left\{ \begin{array}{rl}
          -\Delta v = |\Delta u|, & \quad \textrm{in }\Omega,\\
					v=0, &\quad \textrm{in } \partial \Omega.
					\end{array}\right.			
\end{equation}
By regularity theory, a solution $v$ of \eqref{COMP} is such that $v \in H^2(\Omega)\cap H^1_0(\Omega)$,
and therefore is an admissible candidate for the problem \eqref{CP} with
Navier boundary conditions. On the other hand, we can argue as in the proof of Proposition \ref{positivity} to get by the maximum
principle that $v > |u|$ in $\Omega$. Hence, being $\rho > 0$ a.e. $\Omega$, we have
$$
\dfrac{\int_\Omega(\Delta v)^2}{\int_\Omega\rho v^2} < \dfrac{\int_\Omega(\Delta u)^2}{\int_\Omega\rho u^2}=\Theta_N,$$
which contradicts the minimality of $\Theta_N$. Thus,
$u$ has sign, and so we can take $u>0$ in $\Omega$. This, combined with $-\Delta u \geq 0$ (by maximum principle, being $\Delta^2 u=\Theta_N\rho u>0$ in $\Omega$ and $\Delta u=0$ on $\partial \Omega$),
allows to employ the Hopf Boundary Point Lemma, which gives 
$$
\frac{\partial u}{\partial \nu}< 0\quad\mbox{ on }\partial \Omega.
$$
In order to conclude, it is enough to notice that if $(u,\rho)$ is a CP-optimal pair with Dirichlet boundary conditions, then $\tfrac{\partial u}{\partial \nu}= 0$ on $\partial \Omega$, hence, it {\it cannot} be a CP-optimal pair with Navier boundary conditions as well. 
\end{remark}

\section{A nonlinear eigenvalue minimization problem in conformal geometry}\label{confgeo}
In \cite{Chanillo13}, Chanillo showed the close relation between
a nonlinear eigenvalue minimization problem for the Laplace-Beltrami
operator $-\Delta_g$ and the composite membrane problem. 
More precisely,
let $(\Omega,g_0)$ be a $2$-dimensional bounded Riemannian manifold with
smooth boundary $\partial \Omega$ and consider the conformal class of the metric $g_0$, 
\begin{equation}\label{g0}
[g_0] := \left\{ g \, \textrm{Riemannian metric on } \Omega:\,\exists \,f\mbox{ such that } g = e^{2f} g_0 \right\}.
\end{equation}
Consider another class of Riemannian metrics which is strictly contained in $[g_0]$, 
\begin{equation}\label{ClasseC}
\Ci := \left\{ g \in [g_0]: g \textrm{ satisfies } \eqref{LimC} \, \textrm{and } \eqref{VolC}\right\},
\end{equation}
where
\begin{equation}\label{LimC}
\textrm{there exists a positive constant } A>0 \, \textrm{such that } \|f\|_{L^{\infty}(\Omega)} \leq A;
\end{equation}
and
\begin{equation}\label{VolC}
\textrm{there exists a positive constant } M>0 \, \textrm{such that } \int_{\Omega}dV_g = \int_{\Omega}e^{2f}dV_{g_0} = M.
\end{equation}
The problem is now to minimize the first eigenvalue
of the Laplace-Beltrami operator $-\Delta_{g}$ with Dirichlet boundary conditions, subject to the
constraints provided by the class $\Ci$. In other words,
find a couple $(u,g)$ which realizes 
\begin{equation}\label{ConfMinC}
\inf_{g \in \Ci} \inf_{u \in H^{1}_{0}(\Omega) \setminus \{0\}} \dfrac{\int_{\Omega} \Delta_{g} u \, u \, dV_g}{\int_{\Omega}|u|^2 \, dV_g}.
\end{equation}

In the same
paper, it is raised the question whether similar results can be 
obtained for higher order conformally invariant operators, with special
attention devoted to the {\it Paneitz operator} $P_{n/2}^g$. 
The problem can be stated as follows.
Let $(\Omega,g_0)$ be a $4$-dimensional bounded Riemannian manifold with
smooth boundary $\partial \Omega$. 
Inside the conformal class $[g_0]$, we want to consider the smaller class of Riemannian metrics,
\begin{equation}\label{ClassC}
\Ci := \left\{ g \in [g_0]: g \textrm{ satisfies } \eqref{LimC} \, \textrm{and } \eqref{Vol}\right\},
\end{equation}
where now
\begin{equation}\label{Vol}
\textrm{there exists a positive constant } M>0 \, \textrm{such that } \int_{\Omega}dV_g = \int_{\Omega}e^{4f}dV_{g_0} = M.
\end{equation}
The problem is now to minimize the first eigenvalue
of $P_2^g$ with Dirichlet boundary conditions, subject to the
constraints provided by the class $\Ci$. In other words,
to find a pair $(u,g)$ which realizes 
\begin{equation}\label{ConfMin}
\inf_{g \in \Ci} \inf_{u \in H^{2}_{0}(\Omega) \setminus \{0\}} \dfrac{\int_{\Omega} P^g_{2}u \, u \, dV_g}{\int_{\Omega}|u|^2 \, dV_g}.
\end{equation} 
We stress that the {\it Paneitz operator} $P^g_2$ has a
leading term given by the fourth order differential 
operator $(-\Delta_g)^2$. In particular, 
if we are in the flat case (i.e. $g$ is the standard 
Euclidean flat metric $g_{E}$),
$$P^{g_E}_2 = \Delta^2.$$
In \cite[Proposition 4]{Chanillo13} it has 
been proved that the problem
\eqref{ConfMin} is equivalent to 
\begin{equation}\label{ConfMin2}
\inf_{\rho \in \mathrm{P}_{g_0}} \inf_{u \in H^{2}_{0}(\Omega) \setminus \{0\}} \dfrac{\int_{\Omega} P^g_{2}u \, u \, dV_{g_0}}{\int_{\Omega}|u|^2 \,\rho dV_{g_0}},
\end{equation}
where, for fixed $0<h<H$, $M>0$, we have defined
$$\mathrm{P}_{g_0} := \left\{ \rho:\Omega\to\mathbb R^+ : h \leq \rho \leq H, \, \int_{\Omega} \rho \, dV_{g_0} = M\right\}.$$
For the sake of completeness, we recall here 
a few facts concerning the conformal
change $g = e^{2f} g_0$. The volume forms are related by
$$dV_g = e^{nf} dV_{g_0},$$
\noindent where $n$ is the dimension of 
the Riemannian manifold $\Omega$, namely $4$ in our case.
The Paneitz operator related to $g$ is given by
$$P_{2}^{g}(u) = e^{-4f} P^{g_0}_{2}(u) \quad \textrm{for every } u \in C^{\infty}(\Omega).$$
The first problem is to understand what happens in the flat case,
i.e. for $g_0 = g_E$, where $g_E$ denotes the standard Euclidean metric.
We denote by $dx$ the volume form associated with $g_E$.
We can notice that \eqref{ConfMin2} can be now written as
\begin{equation*}
\inf_{\rho \in \mathrm{P}_{g_0}} \inf_{u \in H^{2}_{0}(\Omega) \setminus \{0\}} \dfrac{\int_{\Omega} (\Delta u)^2 \, dx}{\int_{\Omega}\rho u^2 dx},
\end{equation*}
which coincides with \eqref{CP}.
Therefore we have the following
\begin{theorem}
There exists a pair $(u_{\infty}, \rho_{\infty} g_{0})$ 
which realizes \eqref{ConfMin}. In particular, 
$$\rho_{\infty} = e^{f_{\infty}} = h \, \chi_S + H \, \chi_{S^c},$$
\noindent where 
$$S = \{ u_{\infty}^2 \leq t\} \quad \textrm{for a certain } t > 0.$$
Furthermore, 
$$u_{\infty} \in W^{4,q}(\Omega) \cap C^{3,\gamma}(\overline{\Omega}) \quad \textrm{for every } q\ge1\mbox{ and }\gamma \in (0,1).$$
\end{theorem}


\end{document}